\documentclass[11pt]{article}
\usepackage{stmaryrd}% pour \llbracket et \rrbracket
\usepackage{latexsym}
\usepackage[normalem]{ulem}
\usepackage{srcltx}
\usepackage{amsmath,amsthm,amsfonts,amssymb, mathrsfs, wasysym}
\usepackage[T1]{fontenc}
\usepackage[utf8]{inputenc}
\usepackage{aeguill}% pour les guillemets francais et autres bizrreries du latin1
\usepackage{url}%utile pour les preprints avec url
\usepackage{enumerate}% \begin{enumerate}[cas 1]
\usepackage{mdwlist} % enumerate*
\usepackage{rotating} %utilise pour les fleches \actson \actedon
\usepackage{graphicx}
\usepackage{color}
\usepackage{epsfig}
\usepackage{hyperref}
\usepackage[all]{xy}
\usepackage{xspace}
\newcommand{\stab}{{\rm stab}}

\newtheorem{thm}{Theorem}[section]
\newtheorem{theo}[thm]{Theorem}
\newtheorem{thmbis}{Theorem}
\newtheorem*{thm*}{Theorem}
\newtheorem{notations}[thm]{Notations} 
\newtheorem{dfn}[thm]{Definition} 
\newtheorem{defi}[thm]{Definition} 
\newtheorem*{dfn*}{Definition}
\newtheorem{dfnbis}[thmbis]{Definition}
\newtheorem{cor}[thm]{Corollary}
\newtheorem{coro}[thm]{Corollary}
\newtheorem*{cor*}{Corollary}

\newtheorem{prop}[thm]{Proposition} 
\newtheorem*{prop*}{Proposition} 
\newtheorem*{properties*}{Properties} 
 
\newtheorem{lem}[thm]{Lemma} 
\newtheorem{lemma}[thm]{Lemma} 
\newtheorem*{lem*}{Lemma}

\newtheorem*{claim*}{Claim} 
 
\newtheorem*{fact*}{Fact}

\newtheorem*{qst*}{Question}

\newtheorem*{pb*}{Problem}

\theoremstyle{remark}
 
\newtheorem*{algo*}{Algorithm} 
\newtheorem*{rem*}{Remark}
\newtheorem{rem}[thm]{Remark}

\newtheorem*{example*}{Example}

\newlength{\espaceavantspecialthm}
\newlength{\espaceapresspecialthm}
{\normalfont \vskip \espaceapresspecialthm}

\newenvironment{specialthm*}[1]{
\vskip\espaceavantspecialthm \noindent \textbf{#1} \itshape}%
{\normalfont \vskip \espaceapresspecialthm}
\newcommand{\semidirect}{\ltimes}
\newcommand{\Zmax}{Z_{max}}
\newcommand{\actson}{\curvearrowright}
\newcommand{\ad}{{\rm ad}}
\newcommand{\diam}{\mathop{\mathrm{diam}\,}}

\newcommand {\calC} {{\mathcal {C}}}   
   
\newcommand {\calE} {{\mathcal {E}}}

\newcommand {\calH} {{\mathcal {H}}}

\newcommand {\calP} {{\mathcal {P}}}

\newcommand {\cala} {{\mathcal {A}}}   
   
\newcommand {\calc} {{\mathcal {C}}}

\newcommand {\calf} {{\mathcal {F}}}   
   
\newcommand {\calh} {{\mathcal {H}}}

\newcommand {\calp} {{\mathcal {P}}}   
\newcommand {\calq} {{\mathcal {Q}}}

\newcommand {\bbH} {{\mathbb {H}}}

\newcommand {\bbN} {{\mathbb {N}}}

\newcommand {\bbR} {{\mathbb {R}}}

\newcommand {\bbZ} {{\mathbb {Z}}}

\newcommand {\ie}{i.e.\ }

\newcommand{\es}{\emptyset}
\renewcommand{\phi}{\varphi} 
\newcommand{\m} {^{-1}} 
\newcommand{\eps} {\varepsilon} 

\newcommand {\xra} {\xrightarrow}

\newcommand{\ol}[1]{\overline{#1}}

\newcommand{\normal} {\vartriangleleft}

\newcommand{\dunion}{\sqcup}% on peut utiliser \sqcup ou  \amalg
\newcommand{\ra}{\rightarrow}

\newcommand{\grp}[1]{\langle #1 \rangle}
\newcommand{\ngrp}[1]{{\langle\langle #1 \rangle\rangle}}

\newcommand{\Stab} {{\mathrm{Stab}}}

\newcommand{\id} {\mathrm{id}}

\newcommand{\displ}{\mathrm{displ}}
\newcommand{\depth} {{\mathrm{depth}}}

\newcommand{\PF}{\mathcal{PF}}
\newcommand{\bPF}{\overline{\mathcal{PF}}}
\newcommand{\bG}{\bar{G}}
\newcommand{\bGp}{\bar{G}'}

\newcommand{\DFi}{\bar{G}_i}\newcommand{\DFpi}{\bar{G}'_i}

\title{Recognizing a relatively hyperbolic group by its Dehn fillings}
\author{Fran\c{c}ois Dahmani and Vincent Guirardel \footnote{Both
  authors are supported by the ANR grant  2011-BS01-013 and the Institut universitaire
  de France.
The second author thanks the Centre Henri Lebesgue ANR-11-LABX-0020-01 for its stimulating mathematical research programs.}}

\begin{document}

\maketitle

\begin{abstract}
  Dehn fillings for relatively hyperbolic groups generalize the topological Dehn surgery on a non-compact hyperbolic $3$-manifold such as hyperbolic knot complements. We prove a rigidity result saying that if two non-elementary relatively hyperbolic groups without certain splittings have sufficiently many
isomorphic Dehn fillings, then these groups are in fact isomorphic. Our main application is a solution to the isomorphism problem
in the class of non-elementary relatively hyperbolic groups with residually finite parabolic groups and with no suitable splittings.
\end{abstract}

\section{Introduction}

\subsection{A rigidity result with respect to Dehn fillings}

\paragraph{The general problem.}
Consider the following general problem.
Fix a finitely generated group $G$, and $\calc$ a collection of quotients of $G$.
Can one determine the isomorphism type of $G$ from the isomorphism
classes of the elements of $\calc$ ?

The case where $\calc$ is the
collection of  all finite quotients of a residually finite group $G$ has been widely studied: it amounts to knowing
when two groups with isomorphic profinite completions are actually isomorphic \cite[Corollary 3.2.8]{RiZa_profinite}. 
 Pickel proved a  weak version of this rigidity problem for virtually nilpotent groups: 
there are only finitely many finitely generated virtually nilpotent  groups sharing the same set of isomorphism classes of finite quotients
 \cite{Pickel_fgnilpotent,Pickel_vnilpotent}. 
This was extended by Grunewald, Pickel and Segal to the class of all virtually polycyclic groups \cite{GrPiSe_polycyclic}.
Such a result also holds for  finitely generated virtually free
groups, and  for some   $S$-arithmetic groups
\cite{GrZa_genus,Aka_arithmetic}.

Although it is an easy exercise to check that
finite quotients of a finitely generated abelian group determine its
isomorphism class, theorems saying that $\calc$ determines  $G$ up to
isomorphism are rare.  
  It is an open question due to Remeslennikov  
 whether a finitely generated residually finite group  having the same
 set of isomorphism classes of finite quotients as a finitely generated free group 
has to be free.  
In fact, there exist pairs of non-isomorphic virtually free groups, of nilpotent groups, and even of virtually cyclic groups, 
having isomorphic profinite completions
\cite{Remeslennikov_conjugacy}, \cite{Baumslag_1974}, \cite{GrunewaldScharlau} \cite{GrZa_genus}.
There are also examples in the very rigid context of irreducible lattices in higher rank semi-simple Lie groups
\cite{Aka_profinite}.

\paragraph{Dehn fillings.}
In this paper, we will consider a relatively hyperbolic group $G$, and take for $\calc$ a collection of \emph{Dehn fillings} of $G$.
This notion originates in $3$-manifold theory. In this context, one starts with a complete orientable hyperbolic $3$-manifold $M$ which is
non-compact, but of finite volume and with toral cusps 
(e.g.\ a hyperbolic link complement in $S^3$).
Cutting out the cusps of $M$, one obtains a $3$-manifold $M'$ with toral boundary.
In each boundary torus, choose a simple closed curve. The
   		isotopy classes of these curves in the tori are called slopes.  Given such a  collection $s$ of slopes, 
		one defines  a Dehn filling  $M_s$ of $M$
 		by gluing  solid tori to $M$ to make the slopes in $s$ bound disks. 
Thurston's Dehn filling theorem asserts that 
if one avoids finitely many exceptional slopes  in each torus, all the $3$-manifolds  $M_s$ obtained from $M$ by Dehn filling are hyperbolic 
\cite{Thurston_bullAMS,PetronioPorti,HodgsonKerckhoff_universal}, see also \cite{Lackenby_Dehn,Agol_bounds}. 
Algebraically,
the fundamental group of $M_s$ is isomorphic to the quotient of $\pi_1(M)$ by the normal subgroup
 generated by the cyclic subgroups corresponding to the slopes.

More generally, 
a relatively hyperbolic group $(G,\calp)$      
is a group $G$ together with a finite collection $\calp=\{[P_1],\dots,[P_k]\}$
of conjugacy classes of  finitely generated subgroups $P_i<G$, called \emph{peripheral} subgroups, 
such that $G$ has an action on a Gromov hyperbolic space
that is similar to the action of a finite volume hyperbolic manifold group on the
universal cover $\mathbb{H}^n$, and where the groups $P_i$ are maximal parabolic subgroups 
(see \cite{Bow_relhyp,Farb_relatively} or Definition \ref{def;RHG}).

 In this setting, one  algebraically defines 
a \emph{Dehn filling} of $G$ as any quotient of the form
$\bar G=G/\ngrp{N_1,\dots,N_k}_G$ where  $N_i\normal P_i$ is a normal
subgroup of $P_i$.
The  hyperbolic Dehn filling theorem 
has been extended to this context by 
 Osin \cite{Osin_peripheral} and Groves and
Manning \cite{GroMan_dehn} (and even beyond this class by Osin and the
authors \cite{DGO_HE}).
 It states there is a finite subset $S\subset G\setminus \{1\}$ such that 
if $N_1,\dots,N_k$ don't intersect $S$ 
then $P_i/N_i$ embeds in $\bar G$, and $\bar G$ is hyperbolic relative
to $P_i/N_i$. 
When the hypothesis $(N_1\cup\dots\cup N_k)\cap S = \emptyset$ holds, we say that $\bar G$ is a \emph{proper} Dehn filling.
In particular, if $\bar G$ is a proper Dehn filling of $G$ where each $P_i/N_i$ is finite or virtually cyclic,
then $\bar G$ is  a  hyperbolic group (not relative).
Thus, in the $3$-manifold case, one recovers Gromov hyperbolicity of the fundamental group of the Dehn filled manifold
 except for a finite set of exceptional slopes.
\\

\paragraph{Recognizing a relatively hyperbolic group by its Dehn fillings.}
The question we address in this paper is whether a large enough collection $\calc$ of Dehn fillings of a relatively hyperbolic group $G$
determines $G$ up to isomorphism.

 To build on the example of the beginning of this introduction, one can
consider \emph{finite Dehn fillings}, defined as  Dehn fillings obtained by 
killing finite index subgroups $N_i\normal P_i$ (though $\bar G$ itself remains an infinite group in general).
 The question we ask  is whether one can recover the isomorphism type of $G$ from the
set of isomorphism types of such Dehn fillings.

The earlier discussion  suggests that the class of toral relatively
hyperbolic groups (whose peripheral subgroups are free abelian) is a
good one to consider, and that the class of those with nilpotent
peripheral subgroups might already cause problems.

Maybe surprisingly, our results  apply essentially without assumption on the
peripheral groups themselves apart from residual finiteness.  
In addition to the requirement that $G$ is non-elementary,
the crucial assumption that we make is
the absence of  splittings of a particular kind. 
In some  
sense, this is the generic case (see for instance \cite{DGP_random} saying that random groups do not split).

\begin{thmbis}[Corollary \ref{cor_periph}]\label{thm_main}
Consider $(G,\calp)$, $(G',\calp')$ two groups that are hyperbolic relative to 
infinite, residually finite groups.
Assume that both are non-elementary, and rigid.

Assume that $G$ and $G'$ have the same isomorphism classes of finite Dehn fillings, viewed as groups with a peripheral structure.

 Then $(G,\calp)\simeq(G',\calp')$. 
\end{thmbis}

Let us explain the definitions involved in the statement.
Recall that $\calp=\{[P_1],\dots,[P_k]\}$ is a finite set of conjugacy classes of subgroups of $G$.
 The notation $(G,\calp)\simeq (G',\calp')$ 
means that there is an isomorphism $G\ra G'$ preserving the peripheral structures, \ie sending $\calp$ to $\calp'$.
If $\bar G$
is a Dehn filling of the relatively hyperbolic group $(G,\calp)$, 
it inherits  a peripheral structure $\bar\calp$ which is the
image of $\calp$ under the quotient map 
(this is again a finite set of conjugacy classes of subgroups).
Thus, whether $(\bar G,\bar \calp)\simeq (\bar G',\bar \calp')$ makes
sense for Dehn fillings of $(G,\calp)$ and $(G',\calp')$.

 A subgroup of $G$ is \emph{elementary}
if it is finite, virtually cyclic, or parabolic (\ie contained in a group in a  group $P_i$ up to conjugacy).
In particular, we say that  $(G,\calp)$ is \emph{non-elementary} if $G$ is infinite, not virtually cyclic, 
and $P_i\neq G$ for all $i\in\{1,\dots k\}$.

The crucial non-splitting assumption is the following:
\begin{dfnbis}\label{dfn_rigid}
We say that a relatively hyperbolic group  $(G,\{[P_1],\dots,[P_n]\})$ is \emph{rigid} if 
it has no non-trivial splitting (as an amalgamation, or an HNN extension) 
 over an elementary subgroup 
such that each $P_i$ is conjugate in some factor of the  splitting.
\end{dfnbis}

It is easy to see that one needs
such assumptions:
a free product $G=P_1* P_2$ is hyperbolic relative to the conjugates of $P_1$ and $P_2$, 
and the finite Dehn fillings of $G$ are the groups 
the form $(P_1/N_1)*(P_2/N_2)$ with $N_i\normal P_i$ of finite index. 
Hence, if we take $P,P'$ two non-isomorphic nilpotent groups having the same profinite completion, 
then the relatively hyperbolic groups $P*P$ and $P'*P'$ are non-isomorphic, but 
have  the same collection of finite Dehn fillings.

Theorem \ref{thm_main} is in fact a corollary of a rigidity theorem that does not only apply to \emph{finite} Dehn fillings.

Say that a sequence 
of Dehn fillings $\bar G_i=G/\ngrp{N_1^i,\dots,N_k^i}_G$ of $G$ is
\emph{cofinal}  
if for each $j\leq k$ and all finite subset $A\subset P_j$, $A$ eventually embeds into $P_j/N_1^j$.
In particular, a cofinal sequence of Dehn fillings is eventually proper (see Lemma \ref{lem;cof}).

\begin{thmbis}[see Theorem \ref{thm_DF_carac}]\label{thm_carac_intro}
Consider   $(G,\calp)$,  $(G',\calp')$ two relatively
  hyperbolic groups. Assume that both are non-elementary and rigid, and that each peripheral group is infinite.

Consider cofinal sequences of Dehn fillings $\bar G_i$, $\bar G'_i$ of $G$ and $G'$ respectively, and
assume that there are isomorphisms $\phi_i: (\bar G_i,\bar\calp_i)\xra{{}_\sim }(\bar G'_i,\bar \calp'_i)$.

Then there is an isomorphism  
$\phi:G\ra G'$ sending $\calp$ to $\calp'$.
Moreover, $\phi$ induces infinitely many of the $\phi_i$, up to  inner automorphisms.
\end{thmbis}

In fact, Theorem \ref{thm_main} and \ref{thm_carac_intro} apply under a slightly weaker rigidity assumption,
see Definition \ref{dfn_Zmax_rigid}.
 We also have similar results where the isomorphisms $\phi_i$ are not
assumed to preserve the peripheral structure, but typically, we then need to assume some sort of smallness of
the peripheral subgroups (see Sections \ref{sec_variation}, \ref{sec_DF_carac}).

\paragraph{Uniform geometry of Dehn fillings.}
Our proof is based on a study of the geometry of Dehn fillings, and in particular, on a uniform control
of this geometry among Dehn fillings.
The description we use is very close in sprit to Gromov-Thurston's $2\pi$-theorem
saying that given a  hyperbolic $3$-manifold $M$ with toral cusps and 
$M'\subset M$ obtained by cutting out a horospherical neighbourhood of the cusps,
then for each collection of slopes $s$ whose length is greater than $2\pi$, 
one can put a metric of negative curvature on $M_s$  
that agrees with the initial hyperbolic metric on $M'$
\cite{BleilerHodgson}.

In a similar spirit,  given a cofinal sequence of Dehn fillings $\bar G_n$ of $(G,\calp)$, 
there is 
a sequence of $\bar G_n$-spaces $X_n$ such that:
\begin{itemize*}
\item the hyperbolicity constant of $X_n$ does not depend on $n$, 
\item for each compact set $B\subset X$, and all $n$ large enough, $X_n$ and $X$ isometrically agree on $B$
\item and the actions $G_n\actson X_n$ have uniform properness properties away from the cusps (see Section \ref{sec_lift}).
\end{itemize*}

The proof Theorem \ref{thm_carac_intro} then goes as follows. Consider two cofinal sequences
$\bar G_i,\bar G'_i$ of Dehn fillings of $(G,\calp)$ and $(G,\calp')$,
and a sequence of isomorphisms $\phi_i:\bar G_i\ra \bar G'_i$.
This yields an action of $G$ on the spaces associated to the Dehn fillings of $G'$.

Because these spaces have  a uniform hyperbolicity constant, 
if the minimal displacement of the generators under these actions goes to infinity, Bestvina-Paulin's argument gives us an action on an $\bbR$-tree,
from which one can deduce a splitting contradicting our rigidity assumption.
If on the contrary, the minimal displacement is bounded, but if the point minimally
displaced goes \emph{into the thin part}, 
then this implies that the image of $\phi_i$ lies in the image of a peripheral group of $G'$, which is impossible.
In the remaining case, one can eventually lift infinitely many
isomorphisms $\phi_i$ to a monomorphism $G\ra G'$ thanks to uniform
properness. 
To deduce an isomorphism between $G$ and $G'$, we 
get another monomorphism $G'\ra G$ from the symmetric argument, and one can conclude using a co-Hopf property for rigid relatively hyperbolic groups.

\subsection{A solution to the isomorphism problem}
We now turn to the main application of our rigidity theorem.
 The isomorphism problem is the third algorithmic problem proposed by
 Dehn in the early 1910's, asking for a procedure determining whether
 two given groups are isomorphic.  It has received a certain attention
 for  some classes of negatively curved groups. 
After works of Sela \cite{Sela_isomorphism}, of Groves and the first
author  \cite{DaGr_isomorphism}, and of the authors  \cite{DG2}, a
complete solution      
to this problem for hyperbolic groups is now known. The class of toral relatively hyperbolic groups is also covered by  \cite{DaGr_isomorphism}. A common central feature of  these approaches is the algorithmic study of equations in these groups as a way to control an enumeration of morphisms. 

For groups that are hyperbolic relative to nilpotent subgroups -- a
natural case  
comprising   
the fundamental groups of complete,
finite volume manifolds of pinched negative curvature  -- this approach is probably hopeless:    there exist finitely generated 
nilpotent groups in which
one cannot decide whether a given system of equations has a solution
or not \cite{Romankov_universal}. Here again,  but  
for another reason,  the case of nilpotent peripheral subgroups is a
difficulty (let alone the case of polycyclic groups).

Nonetheless, we obtain a solution in that case.

\begin{thmbis}[see Corollary \ref{cor_vpc}] \label{thm_VPC}  
The isomorphism problem is solvable for non-elementary rigid relatively hyperbolic groups with virtually polycyclic peripheral groups.

More precisely, given two finite presentations of groups $G=\grp{S|R}$, $G'=\grp{S'|R'}$
such that $G$ and $G'$ are hyperbolic relative to some virtually polycyclic groups,
non-elementary, 
and have  no nontrivial splitting over a virtually polycyclic group, 
 one can decide whether $G$ is isomorphic to $G'$.
\end{thmbis}

In particular, it follows from our result that the isomorphism problem
for fundamental groups of complete, finite volume manifolds of
negative pinched curvature is solvable, which was not even known for
complex hyperbolic lattices.
Using a theorem by Farrell-Jones \cite[Corollary 1.1]{FaJo_compact},
this allows to solve the homeomorphism problem for such  manifolds in dimension at least $6$.

In \cite{DaTo_isomorphism}, this result is used to solve the isomorphism problem in  a
class of non-rigid relatively hyperbolic groups, namely for all torsion
free groups that are hyperbolic relative to nilpotent groups.

Theorem \ref{thm_VPC} is actually an incarnation of a much more general
result that we obtain.

  \begin{thmbis}[see Theorem \ref{theo;isom_algo}]\label{thm_main_iso}

    There is an algorithm that solves the following problem.

    The input is  a pair of finitely presented relatively hyperbolic groups $(G,\calp),(G',\calp')$ given by
    finite presentations $G=\grp{S|R},G'=\grp{S'|R'}$
    together with a finite generating set of 
    a representative of  each conjugacy class of peripheral groups. 
We assume that
      \begin{itemize*}
       \item $(G,\calp)$ and $(G',\calp')$ are rigid and non-elementary
      \item peripheral subgroups are infinite, and residually finite.
      \end{itemize*}

    The output is the answer to the question  whether
    $(G,\calP)\simeq(G',\calp')$.  
    If the answer positive, the algorithm also gives an explicit isomorphism.
  \end{thmbis}

The class of virtually polycyclic groups is a very important class of groups 
but the class of finitely presented residually finite groups is much vaster.
It contains many subclasses for which the isomorphism problem is not solvable, and still,
we are able to solve the isomorphism problem with peripheral groups in this class.

Here is a specific example.  
Fix $r,n\geq 2$, and
let $\calC$ be the class of all
semi-direct products $F_r\semidirect \bbZ^n$.  
If $n\geq 4$ and $r\geq 15$, this is a class of residually finite groups for which the isomorphism problem
is unsolvable  \cite{Zimmermann_Klassifikation}. Nevertheless we obtain:

\begin{thmbis}[see Corollary \ref{cor_af}]\label{cor_semidirect} 
 There exists an algorithm that, given presentations for two groups $G,G'$ that are hyperbolic
relative to groups in $\calc$,  non-elementary and rigid, 
 says whether the groups are isomorphic or not.  
\end{thmbis}

\paragraph{Handling peripheral subgroups.} 
  The treatment of the peripheral groups  in
 Theorem \ref{thm_main_iso} is different from the one in 
Theorems \ref{thm_VPC} and \ref{cor_semidirect}. In
 these latter,  
the algorithm is given no specific
 information on them, and 
there is no constraint on the sought isomorphism regarding peripheral
subgroups. 
In Theorem \ref{thm_main_iso}  however, the
 algorithm is given a generating set of the peripheral subgroups, and  
it looks for an isomorphism  
that preserves the peripheral structure.
 
Actually, in many contexts (including
Theorems \ref{thm_VPC} and \ref{cor_semidirect}),   
 a given relatively hyperbolic group has a canonical peripheral
 structure. 
Let us say that a group $H$   is {\it
   universally parabolic} if, for every relatively hyperbolic group
 containing a subgroup isomorphic to $H$, this subgroup is parabolic.
 For instance  all groups in the class $\calC$ of
 Theorem \ref{cor_semidirect}, and all virtually polycyclic groups that are
 not virtually-cyclic are universally parabolic. When a group is
 hyperbolic relative to universally parabolic subgroups, then its
  relatively hyperbolic peripheral structure is canonical. In this
  situation,  one has $G\simeq G'$ if and only if $(G,\calp)\simeq (G',\calp')$.

Then, one needs to find the peripheral subgroups just from a presentation of $G$.
This is possible when the peripheral groups belong to a recursively
enumerable class $\calc$   
 of finitely presented groups that are 
universally parabolic and residually finite
(this asks for the existence of a Turing machine  enumerating
precisely the presentations of the groups in $\calc$).
In this case, one can compute generators of representatives of peripheral subgroups, and
 even presentations of these subgroups, from a presentation of $G$ \cite{DG_presenting}. 
 This allows to prove the following general result, of which
Theorems \ref{thm_VPC} and \ref{cor_semidirect} are corollaries.

\begin{thmbis}[see Theorem \ref{thm_iso_class}]\label{thm_sans_periph} 
 Let $\calc$ be 
any recursively enumerable class of finitely presented groups 
that are residually finite and universally parabolic.

Consider the class $\calh_\calc$ of groups $G$
that admit a structure $(G,\calp)$ of a rigid, non-elementary hyperbolic group  relative to some groups in $\calc$.

Then the isomorphism problem is solvable in the class $\calh_\calc$.
\end{thmbis}

\paragraph{ Why this cannot be used to solve the isomorphism problem
  among residually finite groups.} 
Let us comment on a well known construction that produces rigid
relatively hyperbolic groups with arbitrary peripheral subgroups.   
Given any finitely generated group $P$, one can start from
$P*\bbZ$, then  using small cancellation theory,  
one can construct many  quotients $G_P$ in which $P$ embeds, and which are hyperbolic relative to the
image of $P$, and rigid    
(one can even impose that the quotients  have Kazhdan property $(T)$, hence having
no splitting at all).

One might unreasonably hope that performing this construction with
arbitrary finitely presented residually finite groups $P,P'$
would allow to decide whether they are isomorphic, which is impossible even in restriction to the class $\calc$ of Theorem \ref{cor_semidirect}.
But there is no reason for this to be true,
 as the obtained group $G_P$ 
would highly depend
on which relations are added to construct the small cancellation
quotient, which  in turn  
depends on 
the way $P$ is given to the algorithm, in particular on its generating
set. 
Instead, Theorem \ref{cor_semidirect} 
implies that there exist no computable such construction  that is functorial in $P$, or even that would satisfy $G_P\simeq G_{P'}$ whenever $P\simeq P'$.

We also prove a negative result showing that one cannot hope to generalize Theorem \ref{thm_VPC}
too much:
the isomorphism problem is not solvable 
in the class of rigid relatively hyperbolic groups that are hyperbolic relative to finitely presented solvable groups (see Proposition \ref{prop_solvable}).

\paragraph{How to solve the  isomorphism problem.}
Let us present    
the proofs of our results. 
As mentioned above, the usual approach to the isomorphism problem based on the solution of equations 
is hopeless as soon as
one allows arbitrary finitely generated nilpotent groups as parabolics.
Our rigidity theorem makes it possible to use a completely different strategy:
we look for a
proof that groups are not isomorphic by investigating whether their
collection of finite Dehn fillings are different.

Indeed, since there is a general algorithm that stops if and only if two finite presentations represent the same group,
one has to find an algorithm that stops if and only if the two given groups are non-isomorphic.

In order to do so, enumerate a canonical infinite list of finite Dehn fillings $\bar
G_i,\bar G'_i$ of $G$ and $G'$, by quotienting by  a characteristic
subgroups of finite index in the peripheral groups. Our rigidity theorem says that if $(G,\calp)$ is not isomorphic to $(G',\calp')$, then 
for infinitely many indices $i$, $\bar G_i$ will not be isomorphic to $\bar G'_i$, preserving the peripheral structure.
Also, for $i$ large enough, $\bar G_i$ and $\bar G'_i$ are hyperbolic
groups (with a lot of torsion).  From 
our solution to the isomorphism problem for hyperbolic groups \cite{DG2}, 
we have an algorithm that will say that $\bar G_i$ is not isomorphic to $\bar G'_i$ (preserving the peripheral structure),
thus certifying that $(G,\calp)\not\simeq (G',\calp')$.

The reader can see that our strategy  relies on the existence of a solution to the isomorphism problem
for all hyperbolic groups. Let us stress that  any solution would be appropriate.
Thus, although the solution in \cite{DG2} is the only one available at the moment, 
the reader who is ready to take it as a black box needs not be
familiar with  its proof.

\vskip.5cm
\subsection{Acknowledgments}

The authors would like to thanks T. Delzant, M. Sapir, O. Kharlampovich for stimulating discussions.
%\Vmodif Both authors acknowledge support from ANR grant ANR-11-BS01-013 and from the Institut universitaire de France.

\tableofcontents

\section{Relatively hyperbolic groups}
\label{sec_RH}

Let $X$ be a locally compact $\delta$-hyperbolic space.  
 Up to increasing $\delta$, we may assume that any triangle $(x_1,x_2,x_3)$ 
has a \emph{center} $c$ at distance at most $\delta$ from any geodesic $[x_i,x_j]$ with $i\neq j\in \{1,2,3\}$.
Recall from \cite[Section 2]{Hruska_quasiconvexity} that a \emph{horofunction} based at $\zeta\in \partial_\infty X$
is a function $h:X\ra \bbR$ such that there is a constant $D_0$ such that 
for all $x,y$, and any center $c$ of the triangle $(x,y,\zeta)$,
one has $$|h(x)-h(y)-d(y,c)+d(x,c)|\leq D_0.$$
A \emph{horoball} centered at $\zeta$ is a subset $B\subset X$ such that there exists $D_1> 0$ with
$$h\m([D_1,\infty)\subset B\subset h\m\big([-D_1,\infty)\big).$$

\begin{defi}\label{def;RHG}
  A pair  $(G,\calP)$ is a \emph{relatively hyperbolic group} if $G$ is a finitely generated group and 
  $\calP$ is a collection of subgroups of $G$ closed under conjugation, such that there exists a proper geodesic hyperbolic space $X$ on which $G$ acts by isometries, and a  $G$-invariant collection $\calH$ of disjoint horoballs in $X$ such that
\begin{itemize}
  \item $G$ acts co-compactly on the complement of $\calH$ in $X$.
  \item The map sending a horoball in $\calH$ to its stabilizer, is a bijection $\calH \to \calP$.
\end{itemize}
\end{defi}

This definition was introduced by Bowditch, proved equivalent to other definitions in the literature, and studied by many authors
\cite{Bow_relhyp,Yaman_relhyp,DrutuSapir_tree-graded,Osin_relatively,Hruska_quasiconvexity}.

The groups of the family $\calP$ are called the \emph{peripheral subgroups} of $(G,\calP)$, 
and their subgroups are simply called \emph{parabolic subgroups}.
It is well known that $\calP$ is the union of finitely many conjugacy classes of subgroups $P_1,\dots,P_n$. 
Denoting by $[P_i]$ the conjugacy class of $P_i$,
we often view $\calp$ as a set of conjugacy classes $\{[P_1],\dots,[P_n]\}$,
and we say that $G$ is hyperbolic relative to $P_1,\dots,P_n$.
Still abusing notation, we denote by $X\setminus\calh$ the complement of the union of the horoballs.

For technical reasons, we will additionally assume, without loss of generality, that $X$ is a metric graph whose
edges have the same positive length. 
Up to rescaling if necessary, we can assume that the hyperbolicity constant is as small as we want.

We say that an element $g\in G$ is hyperbolic 
if it has infinite order and is not  in any of the peripheral subgroups.
 Equivalently, $g$ is a loxodromic isometry of the space
$X$.

Recall that a subset $B\subset X$ is $C$-quasiconvex if for any $x,y\in B$,
any geodesic $[x,y]$ lies in the $C$-neighbourhood of $B$.
We say that $B$ is $C$-\emph{strongly quasiconvex} if for any two points $x,y\in B$, there exists $x^\prime,y^\prime\in B$
and geodesics $[x,x'],[x',y'],[y',y]$  
  such that $\max \{ d (x, x^\prime ), d (y, y^\prime )\} \le C $ and $[x,x']\cup[x',y']\cup[y',y]\subset B$.
It is well known that if $B$ is $C$-quasiconvex, then for all $D\geq C$, its $D$-neighbourhood 
$B^{+D }$ is $2\delta$-strongly quasiconvex (see for instance \cite[Lemma 3.4]{DGO_HE}).

The following well known lemma says that one can assume that the horoballs of $\calh$ 
are given as the superlevel sets
of an \emph{invariant} family of horofunctions, and that these superlevel sets are $2\delta$-strongly quasiconvex.

\begin{lem}\label{lem_eqv_horo} Let $(G,\calp)$ be relatively hyperbolic, and $X$ be a hyperbolic space as in Definition \ref{def;RHG}, 
and $p\in X\setminus\calh$ a base point.

Then for each peripheral subgroup $P\in\calp$, 
one can assign a $P$-invariant horofunction $h_P:X\ra \bbR$
satisfying the following equivariance property
$$\forall P\in \calp,\forall x\in X,\ h_{gPg\m}(x)=h_P(g\m x)$$ 
and such that the following holds.

For any $R\geq 0$, consider the family of horoballs $\calh_R=\{h_P\m([R,\infty))|P\in \calp\}$.
Then
\begin{enumerate}\item 
  each horoball $h_P\m([R,\infty))$ is $2\delta$-strongly quasiconvex,
\item 
  $\calh_R$ is $R$-separated: any two points in two distinct horoballs
  are at distance at least $R$.
\item \label{it_basepoint} all horoballs in $\calh_R$ are disjoint from $B(p,R)$.
\item $\calh_R$ satisfies Definition \ref{def;RHG}.
\item \label{it_horo_sep} there exists a constant $D_{\ref{lem_eqv_horo}}\geq 0$ such that for all $R'\geq R+D_{\ref{lem_eqv_horo}}$, the $(R'-R-D_{\ref{lem_eqv_horo}})$-neighborhood
of each horoball of $\calh_{ R'}$ is contained in the corresponding horoball of $\calh_R$.
\end{enumerate}
\end{lem}

\begin{proof}
Let $\calh$ be a system of horoballs as in Definition \ref{def;RHG},
and consider a horoball $B\in \calh$ with stabilizer $P$.
Let  $h$ be a horofunction and $D_1\geq 0$ be such that
$$h\m([D_1,\infty))\subset B\subset h\m([-D_1,\infty)).$$
It follows from the definition of a horofunction that 
there exists a constant $D_2$ such that for any $R$,
the horoball $h\m([R,\infty))$ is $D_2$-quasiconvex.
Moreover, $h$ is coarsely $1$-Lipschitz: 
there exists a constant $D_3\geq 0$ such that $|h(x)-h(y)|\leq d(x,y)+D_3$.

Since $B$ is $P$-invariant, $|h(px)-h(x)|$ is bounded independently of $x\in X$ and $p\in P$.
Therefore, the function $h_1$ defined by $h_1(x)=\sup_{g\in P} h(gx)$ is a well defined $P$-invariant horofunction on $X$.
Let $D$ be such that for all $R\in \bbR$, the horoball $h_1{}\m([R,\infty))$ is $D$-quasiconvex.
The horofunction $h_2(x)=\inf \{h_1(y)| d(x,y)\leq D\}$ 
is such that the horoball $h_2{}\m([R,\infty))$ is the $D$-neighbourhood of $h_1\m([R,\infty))$,
and is therefore $2\delta$-strongly quasiconvex for all $R$.
Note that there exists $K_1,K_2$ such that for all $x\in X$, $|h_2(x)-h(x)|\leq K_1$
and $|h_2(x)-h_2(y)|\leq d(x,y)+K_2$. 

Define $h_P(x)=h_2(x)-K_1-K_2-D_1$.
Then $h_P$ is $P$-invariant, and all its superlevel sets are $2\delta$-strongly quasiconvex.
We claim that the $R$-neighbourhood $B_R^{+R}$ of the horoball $B_R=h'{}\m([R,\infty))$ is contained in $B$.
Indeed, for each $x\in B_R^{+R}$, consider $y\in B_R$ with $d(x,y)\leq R$ then
$h(x)\geq h_2(x)-K_1 \geq h_2(y)-R-K_2-K_1 =h'(y)-R +D_1 \geq R-R+D_1 =D_1$,
so $x\in h\m([D_1,\infty))\subset B$.
This proves the claim, and in particular that $B_0\subset B$.
Note that conversely, for each $R$, there exists $L_R$ such that 
$B$ is contained in some neighbourhood $B_R^{+L_R}$ of $B_R$.

Defining such a horofunction $h_P$ for a representative of each conjugacy class
 in $\calp$, and extending by equivariance,
one gets an equivariant family of horofunctions $h_P$.
The inclusion $B_0\subset B$ ensures the disjointness of the horoballs in $\calh_0$,
and therefore in $\calh_R$,
and cocompactness of $X-\calh_R$ follows from the fact that $X$ is proper and that
$X-\calh_R$ is contained in the $L_R$-neighbourhood of $X-\calh$.
It follows that $\calh_R$ still satisfies Definition \ref{def;RHG}.
Since the $R$-neighbourhood of $B_R$ is contained in $B$, this implies that
$\calh_R$ is $2R$-separated, and does not intersect $B(p,R)$.

The last assertion clearly holds for $D_{\ref{lem_eqv_horo}}=K_2$ since
$|h_P(x)-h_P(y)|\leq d(x,y)+K_2$. 
\end{proof}

From now on, we assume that each horoball $B$ in our system of horoballs $\calh$ 
is defined by a horofunction $h_B$ as in the lemma above.

The following lemma is useful. It follows for instance from
\cite{Tukia_convergence}, but we propose a proof for completeness.

\begin{lem}\label{lem_loc_parab} \label{lem_tits}
  Let $(G,\calp)$ be a relatively hyperbolic group. Let $H<G$ be a
  subgroup (maybe not finitely generated). 

If $H$  contains no hyperbolic element, then $H$ is parabolic or finite.

If $H<G$ contains a hyperbolic element, then it is either virtually
cyclic, or it contains  a pair of hyperbolic elements
generating a free subgroup  all whose the non-trivial elements are
hyperbolic.
\end{lem}

\begin{proof}
 Consider $R=1000\delta$ and $\calh_R$ the corresponding system of horoballs.
Let $K$ be the maximal cardinality of a set of elements of $G$ that move a point of $X\setminus \calh_R$ by at most $100\delta$
(this exists by proper cocompactness).

Assume that $H$ is infinite, without hyperbolic element. Consider $g_1,\dots,g_n,\dots \in H$ an enumeration of $H$, and consider $n>K$.
By \cite[Proposition 3.2]{Koubi_croissance},
there is a point $p\in X$ that is moved by at most $100\delta$ by all the elements $g_1,\dots,g_n$.
Since $n>K$, $p$ lies in a horoball $B\in \calh_R$.
Since horoballs are $1000\delta$-separated,  $\grp{g_1,\dots,g_n}$ preserves $B$ and its center $\omega\in \partial_\infty X$.

No other horoball is preserved by $\grp{g_1,\dots,g_n}$ since 
otherwise, there would exist some bi-infinite geodesic
$(\omega,\omega')$ whose points are moved by at most  $50\delta$ by
elements of $\grp{g_1,\dots,g_n}$, and that would be a contradiction since this geodesic intersects  $X\setminus \calh_R$.

It follows that $B$ is the only horoball of $\calh_R$ preserved by $\grp{g_1\dots,g_n}$.
Therefore, $B$ does not depend on $n$, so $H$ is parabolic.

If now $H$ contains two hyperbolic elements with disjoint pairs of
attracting and repelling points in the boundary $\partial X$, then a
standart ping pong argument in $\partial X$ shows that some powers of
these elements generate a free group  all whose non-trivial elements are
hyperbolic.

Next, if  $g_1, g_2$  are two hyperbolic elements with one common
fixed point in the boundary, 
by properness of the
action of $G$ on $X$, they
must have two common fixed points in the boundary.
Indeed, in such case, the two quasi-invariant axes of the elements will have infinite rays remaining
close to each other, thus forcing
the commutators of powers $g_1^k$ and $g_2^m$, for arbitrary $k, m$ of
a certain sign, 
to almost fix the starting
point of these rays. Properness implies that  some powers of $g_1$ and
$g_2$ commute. Therefore,   $g_1, g_2$ have the same
fixed points in the boundary. 

Finally, if all hyperbolic elements of a
subgroup $H$ have the same
pair of fixed points in $\partial X$, by properness of the action on
the hull of these two points, one gets that $H$ is
virtually cyclic. 
\end{proof}

\section{Dehn fillings}  
\label{sec_DF}
\subsection{Coning-off horoballs at different depths}
\label{sec_cone}

We now recall the construction of a \emph{parabolic cone-off} performed in \cite[\S 7.1]{DGO_HE}. Similar constructions can be found, or at least easily obtained, from the work of Groves and Manning \cite{GroMan_dehn},  see also \cite{Osin_peripheral}.

  Let $(G,\calP)$ be a relatively hyperbolic group.
Choose $X$ a proper $\delta$-hyperbolic metric graph on
which  $(G,\calP)$ acts as   
in
Definition \ref{def;RHG}. We will assume that $\delta\leq\delta_c$ 
where $\delta_c$ is a universal constant defined by Proposition \ref{prop_coneoff} below;
this is no loss of generality since one can rescale (once and for all) the metric on $X$
to achieve this.

Following \cite[\S 7.1]{DGO_HE}, we are now going to describe a family of spaces  $\dot X_R$, indexed by $R\geq 0$.
This is done in several steps.
For each $R\geq 0$ consider $\calH_R$ the family of horoballs defined in Lemma \ref{lem_eqv_horo}.

We fix some number $r_U>0$ (its value will be given by Proposition \ref{prop_coneoff} below).
For each horoball $B\in \calH_R$, we consider  a cone $C(B) \simeq
B\times [0,r_U]/\sim$  where $\sim$ is the relation $(x,r_U) = (y,r_U)$ 
  for all $x,y\in B$. The image of $B\times \{r_U\}$ is a point called the apex of the cone, we denote it by $c_B$.
For each edge $e$ of $B$, the image of  $e\times[0,r_U]$ is a triangle of $C(B)$, 
and we put on it a  metric that makes it isometric to a sector in $\bbH^2$ centered at 
the apex, of radius $r_U$, and whose arclength coincides with the length of $e$ (see \cite[\S 5.3]{DGO_HE} or \cite[section I.5]{BH_metric} for details).

Then consider $C_{\calH_R}(X)=\left( X\dunion \coprod_{B\in\calh_R} C(B)\right)/{\sim}$
obtained from $X$ by gluing, for each horoball $B\in\calh_R$, the cone $C(B)$ on $X$ using 
the identification $(b,0)\sim b$ for all $b\in B$.
The space $C_{\calh_R}(X)$ is endowed with the corresponding path metric. 
  Note that $G$ acts naturally by isometries
  on $C_{\calH_R}(X)$, and preserves the family of cones $C(B), B\in \calH_R$. 

By \cite{DeGr_mesoscopique} (see also \cite[\S 5.3]{DGO_HE}), 
there exists $\delta_c>0$ and $r_U>0$ such that if $X$ is $\delta_c$-hyperbolic,
$C_{\calH_R}(X)$ is a geodesic hyperbolic space, 
  whose hyperbolicity constant does not depend of $R$. 

Following \cite[Definition 7.2]{DGO_HE}, we are going to recall the construction of the parabolic cone-off $\dot X_R\subset C_{\calH_R}(X)$.
It is a specific thickening of the cone on the horospheres of $X$.
The properties of $\dot X_R$ that we will use are captured in Proposition \ref{prop_coneoff} below.

To explain the construction of $\dot X_R$, 
fix a horoball $B\in \calh_R$, and consider $\partial B\subset X$ its topological boundary in $X$.
For all $x,y\in \partial B$ with $d_{B}(x,y)< \pi\sinh r_U$, and each geodesic $[x,y]_B\subset B$,
the set of points in $C(B)$ whose radial projection lies in $[x,y]_B$ is isometric to a sector in $\bbH^2$
of angle $\frac{1}{\sinh r_U}d_{B}(x,y)<\pi$. Then we consider in this sector the filled hyperbolic triangle $T_{x,y}\subset C(B)$ 
bounded by the three geodesics $[c_B,x]_{C(B)}\cup [c_B,y]_{C(B)}\cup [x,y]_{C(B)}$.
We denote by $\dot B\subset C(B)$ the union of all such triangles $T_{x,y}$,
where $(x,y)$ varies among all pairs of points in $\partial B$ with $d_{B}(x,y)< \pi\sinh r_U$.
Then we define $\dot X_R\subset C_{\calH_R}(X)$ as the union of $X\setminus \calh_R$ with
$\bigcup_{B\in \calh_R} \dot B$.

We endow $\dot X_R$ with the path metric induced by $C_{\calH_R}(X)$,
 and call it the \emph{cone-off} of $X$ at depth $ R$. 
By \cite{DGO_HE} (see discussion before Definition 7.2), 
$\dot{X}_R$ is a geodesic hyperbolic metric space with hyperbolicity constant independent of $R$.
It is locally compact, except at the neighbourhood of apices.

  In fact, one can think of  $\dot{X}_R $ as a cone-off 
  over the horospheres  of $X\setminus \calH$, since points sufficiently deep in the horoballs 
   of $\calH$ are not in  $\dot{X}_R$. 
  One can therefore think of this construction as a 
  balance between Bowditch's horoballisation of the Cayley graph of $G$ 
  (which would be $X \simeq\dot{X}_\infty $), 
  and Farb's cone-off (which would be  $\dot{X}_0$). A similar space can be constructed following Groves and 
  Manning's construction of the space with horoballs \cite{GroMan_dehn}, and coning-off 
  (as Farb does in \cite{Farb_relatively})  at depth $R$. 
\\

We denote by $\iota_R$ the inclusion $\iota_R:X\setminus \calh_R \subset \dot X_R$.
It is obviously injective, 
but it is not isometric. 
However, if $p\in X$ and $r>0$ are such that $B_X(p,r)\subset X\setminus \calh_R$,
then $\iota_R$ clearly induces an isometry $B_X(p,r/2)\ra B_{\dot X_R}(\iota_R(p),r/2)$.

  Summing up the properties of this space:

  \begin{prop}[Cone-off at depth $R$, {\cite[\S 7.1 \& Proposition 7.5]{DGO_HE}}]\label{prop_coneoff}
There exist universal constants $\delta_c,r_U,\delta_0$ such that the following holds.
Let $(G,\calP)$ be a relatively hyperbolic group, and $X$ a $\delta_c$-hyperbolic graph 
  as in Definition \ref{def;RHG}.    
  For all $R\geq 0$,  
  consider the family of horoballs $\calh_R$ as in Lemma \ref{lem_eqv_horo}, 
and $\dot X_R$ the corresponding parabolic cone-off.

Then $\dot{X}_{R}$ is a geodesic $\delta_0$-hyperbolic space, with an 
isometric and  co-compact (although not proper)  $G$-action.

Moreover, for all $x\in X$ and $r>0$ such that $B_X(x,r)\cap \calh_R=\es$, 
the injective equivariant map  $\iota_R: X\setminus \calh_R \to  \dot{X}_{R}$ 
induces an isometry $B_X(x,r/2) \to B_{\dot X_R}(\iota_R(x),r/2)$.

  \end{prop}

For each $R'\leq R$ and each horoball $B'\in \calh_{R'}$,
it is convenient to consider the trace $\dot B'$ of the horoball $B'$ in $\dot X_R$ defined as follows:
 $\dot B'=(B'\cup C(B))\cap \dot X_R$ where $B$ is the unique horoball of
$\calh_R$ contained in $B'$, and $C(B)$ is the hyperbolic cone on $B$ defined above.
The family of all these subsets $B'$ of $\dot X_R$ is denoted by $\calh_{R'}^{\dot X_R}$.

\begin{lem}\label{lem_dotsep}
For all $R'\leq R$, $\calh_{R'}^{\dot X_R}$ is $R'$-separated in $\dot X_R$.
\end{lem}

\begin{proof}
Consider $x_1\in\dot B_1,x_2\in\dot B_2$ with 
$\dot B_1\neq \dot B_2\in \calh_{R'}^{\dot X_R}$
and $d_{\dot X_R}(x_1,x_2)$ close to the infimum.
Let  $\gamma$ be a geodesic of $\dot X_R$ 
joining $x_1$ to $x_2$.
Up to changing $\gamma$ to a subsegment, and changing $x_1$, $x_2$ accordingly, 
we can assume that $\gamma\cap \dot B_1=\{c_1\}$, $\gamma\cap \dot B_2=\{c_2\}$
and $\gamma\cap \dot B\neq \es$ for every other $\dot B\in \calh_{R'}^{\dot X_R}\setminus\{\dot B_1,\dot B_2\}$.
If follows that $\gamma$ is a path contained in $X$, whose endpoints lie in distinct horoballs
of $\calh_{R'}$. Since the length of $\gamma$ is $d_{\dot X_R}(x_1,x_2)$, we get that 
$d_X(x_1,x_2)\leq d_{\dot X_R}(x_1,x_2)$.
Since $\calh_{R'}$ is $R'$-separated, we get that $R'\leq d_X(x_1,x_2)\leq  d_{\dot X_R}(x_1,x_2)$.
\end{proof}

  \subsection{Dehn filling and Dehn kernels}

\newcommand{\DF}{\bar G}

\begin{defi}  
A \emph{Dehn kernel} of a relatively hyperbolic group
$(G,\{[P_1],\dots,[P_k]\})$ is a normal subgroup $$K=\ngrp{N_1,\dots,N_k}_G\normal  G$$    of $G$ normally
generated by a collection of subgroups $\{N_1, \dots, N_k\}$ where each $N_j$ is a
normal subgroup of $P_j$.

A \emph{Dehn fillling} of  a relatively hyperbolic group
$(G,\{[P_1],\dots,[P_k]\})$ is the quotient $\DF = G/K$  of $G$ by
a Dehn kernel $K$. We also say that $\bar G$ is the Dehn filled group of $G$
by $K$. 

Let   $X$ be a $\delta_c$-hyperbolic space for $G$, as in Definition
\ref{def;RHG},   and for  $R>0$, let $\calh_R$ be an invariant family of
$R$-separated  horoballs as
in Lemma \ref{lem_eqv_horo}, 
       and $\dot X_R$ the corresponding cone-off
       as in Proposition \ref{prop_coneoff} above. The \emph{Dehn
         filled space}  of  $\dot X_R$ by a Dehn kernel $K$  is the
       space  $\bar X_R = \dot X_R/K$.

\end{defi}

The Dehn filled  space $\bar X_R$ actually depends on $R$ and $K$. 
From the context, it will always be clear what $K$ is, so we do not usually
specify it in the notation.

  \begin{theo}[Dehn filling, {\cite[Proposition 7.7, Corollary 7.8]{DGO_HE}}]\label{theo;VRF}    There exists a universal constant $\delta_1$ such that the following holds.
       Let $(G,\{[P_1],\dots,[P_k]\})$ be a relatively hyperbolic
       group,  $X$ a $\delta_c$-hyperbolic space for $G$  
as in Definition
\ref{def;RHG}, 
       $p \in X$ a base point, and for each $R>0$, $\calh_R$ a family of horoballs, 
       and $\dot X_R$ the corresponding cone-off
       as in Proposition \ref{prop_coneoff} above.

       Then for each $R>0$, there exists a finite subset $S_R\subset G\setminus\{1\}$ such that the following holds.

       For each $j\in\{1,\dots k\}$, consider $N_j\normal P_j$ a normal 
       subgroup of $P_j$ avoiding $S_R$, and denote by
       $K=\ngrp{N_1,\dots,N_k}\normal G$  the associated Dehn kernel. 
         Then
     \begin{enumerate*}
       \item \label{it_hyp} the Dehn filled space $\bar X_R=\dot{X}_R/K$ is $\delta_1$-hyperbolic,           
     the Dehn filled group $G/K$   
      acts discretely co-compactly 
     by isometries on $\bar X_R$; under the additional assumption that $N_j$ has finite index in $P_j$, then
     $\bar X_R$ is proper, and $G/K$ acts properly discontinuously on $\bar X_R$, and in particular, $G/K$ is a hyperbolic group;
       \item \label{it_finite} $K\cap P_j = N_j$ for all $j$,
        \item if $x\in X$ and $r>0$ are such that $B_X(x,r)\subset X$ is disjoint from $\calh_R$, then
          \begin{enumerate*}
          \item \label{it_isometry}  the quotient map  
        $\pi_K:\dot{X}_{R}\to  \bar{X}_R$ restricts to an isometry
        $B_{\dot X_R}(x,r/2)\ra B_{\bar X}(\pi_K(x),r/2)$
      \item \label{it_disjoint} for all $g\in K\setminus 1$, $g B(x,r)\cap B(x,r)=\es$.
          \end{enumerate*}       
      \end{enumerate*}
  \end{theo}

  \begin{rem}

    One can be slightly more precise on how to take $S_R$.   
 We may take for $S_R$ a set of representatives of conjugacy classes of elements in 
$\cup_j P_j\setminus \{1\}$ that move some point of $X\setminus
\calh_R$ by at most some specific number (namely $4\pi\sinh(r_U)$, a
true constant, see Lemma 7.5 and Proposition 7.7 in \cite{DGO_HE}).
In particular, we can impose that $S_{R'} \subset S_R$ when $R'\leq R$.

In fact, $G/K$ is relatively hyperbolic relative to the image of $\calp$ in $G/K$ \cite{GroMan_dehn,Osin_peripheral}.
  \end{rem}

\subsection{Proper Dehn kernels, and depth of subgroups}
The condition for Theorem \ref{theo;VRF} to apply is that $N_j\cap S_R=\es$.
The larger $R$ is, the more the space $\bar X_R$ looks like $X$.
It is therefore convenient to define the \emph{depth} of a subgroup of
some $P_i$ as the largest possible such $R$ for which the theorem promises a 
$\delta_1$-hyperbolic Dehn filled space $\bar X_R$, for this
particular subgroup. However, it will be convenient to extend this
definition to  any subgroup of $G$, as follows.

  \begin{dfn}[Depth of a subgroup]\label{dfn_depth}
Given a subgroup $H<G$, we define its \emph{depth} in $\bbN\cup \{\pm\infty\}$ as
$$\depth(H)=\sup \{R\in\bbN\setminus\{0\} | H\cap S_R=\es\}.$$ 
  \end{dfn}

In particular, $\depth(H)=-\infty$ if there is no $R>0$ such that $H\cap S_R=\es$
and $\depth(H)=+\infty$ if $H\cap S_R=\es$ for all $R$. 
And if $H$ contains a parabolic element, then $\depth(H)<+\infty$.

  \begin{dfn}\label{dfn_dehn}

A Dehn kernel $K=\ngrp{N_1,\dots,N_k}_G\normal G$ such that $N_j$ has
positive depth for all $j$ is called a \emph{proper} Dehn
kernel.
  \end{dfn}

Thus, a Dehn kernel is proper if and only if Theorem \ref{theo;VRF}
applies. In particular, for a proper Dehn kernel $K=\ngrp{N_1,\dots,N_k}_G$, one
has $N_i=K\cap P_i$. Since 
$S_R\subset P_1\cup\dots\cup P_k$, we get that 
$\depth(K)=\min\{\depth(N_1,\dots, N_k)\}$.
In other words, $K$ is a Dehn kernel of depth $\geq R$ if and only if $\depth(N_j)\geq R$ for all $j$.

An application of Theorem \ref{theo;VRF} also gives the following.

  \begin{cor}\label{cor_hyp}
    If $K=\ngrp{N_1,\dots,N_k}_G$ is a proper Dehn kernel of $G$, and if $N_j$ has finite index in $P_j$ for all $j$,
then $G/K$ is Gromov-hyperbolic.
  \end{cor}

\subsection{Traces of horoballs in the Dehn filled space}

  Given $R'\leq R$, and $\dot B\in \calh^{\dot X_R}_{R'}$ 
 the trace of a $\calh_{R'}$-horoball (as defined before Lemma \ref{lem_dotsep}),
  we define $\bar B=\pi_K(\dot B)$ its image in $\bar X_R$.
  We also denote by $\bar \calh_{R'}^{\bar X_R}$ the set all of such $\bar B\subset \bar X_R$.

\begin{lem}\label{lem_barsep}
For all $R'\leq R$, the family of subsets $\bar \calh_{R'}^{\bar X_{R}}$ is $R'$-separated,
and $\diam(\bar X_R\setminus \bar\calh_{R'}^{\bar X_R})/\bar G\leq \diam(X\setminus
\calh_{R'})/G$. 

The stabilizer of $\bar B$ is the image in $\DF$ of the peripheral subgroup $P\in \calp$ stabilizing the corresponding horoball of $X$.
\end{lem}

\begin{proof}
  Consider $\bar B_1\neq \bar B_2\in\bar\calh_{R'}^{\bar X_R}$,
  and $\bar x_1\in\bar B_1,\bar x_2\in \bar B_2$. Let $\bar c$ a geodesic path
  joining them. Lift $\bar c$ to a path $c\subset \dot X_R$ joining $\dot B_1$ to $\dot B_2$.
  By Lemma \ref{lem_dotsep}, $c$ and hence $\bar c$ have length at least $R'$.
  The rest of the statement is clear.
\end{proof}

\subsection{Lifting elements and uniform properness}\label{sec_lift}
  We now collect  handy consequences of Proposition \ref{prop_coneoff} and  Theorem \ref{theo;VRF}.

  \begin{lemma}[Lifting elements]\label{lem;lift}
    Assume that  $K$ is a proper Dehn kernel of $G$ with  $\depth(K)\geq R$. 
    Consider $x\in X$ and $r>0$ such that $B_X(x,r)\cap \calh_R=\es$,
    and $\bar x=\pi_K\circ \iota(x)$  the image of $x$ in $\ol X_R$.

    Then $\pi_K\circ \iota_R$ restricts to an isometry between 
$B_X(x,r/2) \to B_{\bar X_R}(\bar x,r/2)$.

Moreover, for any $\bar{g}\in G/ K$  such that   
  $d_{\bar{X}} (\bar x, \bar g \bar x) \leq r/2$,
  there exists a unique preimage $g\in G$ of $\bar g$ such that  $d_X(x, g x) \leq  r/2$,
  and this preimage satisfies $d_X(x,gx)=d_{\bar X_R}(\bar x,\bar g\bar x)$.
  \end{lemma}

  \begin{proof}
    Let $g\in G$ be any preimage of $\bar g$.
By Proposition \ref{prop_coneoff} and Assertion (\ref{it_isometry}) of Theorem \ref{theo;VRF},
the map  $\pi_K\circ\iota_R: X\setminus\calh_R \to  \bar{X} $ 
induces an isometry $B_X(x,r/2) \to B_{\bar X_R}(\bar x,r/2)$. Consider
 $x'\in B_X(x,r/2) $ the preimage of $\bar g \dot x$ under this isometry. 
Consider any preimage $g\in G$ of $\bar g$.
Since $\pi_K\circ\iota_R(gx)=\bar g x=\pi_K\circ \iota_R(x')$,
there exists $k\in K$ such that $\iota_R(gx)=k\iota_R(x')$,
\ie $x'=k\m gx$. This shows the existence.

If $g_1,g_2\in G$ both project to $\bar g\in G/K$, and if $d_X(x,g_1x)\leq r/2$ and $d_X(x,g_2x)\leq r/2$,
then the element $k=g_1g_2\m$ is an element of $K$ such that  
$k B(x,r/2)\cap B(x,r/2)\neq \es$.
By Assertion (\ref{it_disjoint}) of Theorem \ref{theo;VRF}, $k=1$ and uniqueness follows.
  \end{proof}

By Lemma \ref{lem_eqv_horo}(\ref{it_basepoint}), the lemma applies to the base point $p$ for $r=R$.
We denote by $\bar p=\pi_K(p)$ the image of the base point $p$ in $\bar X_R$.

We thus get:

  \begin{cor}\label{cor_lift}
    If $K$ is a proper Dehn kernel of $G$ with $\depth(K)\geq R$.
    Then $\pi_K\circ \iota_R$ restricts to an isometry between 
$B_X(p,R/2) \to B_{\bar X_R}(\bar p,R/2)$.

Moreover, for any $\bar{g}\in G/ K$  such that   
  $d_{\bar{X}} (\bar p, \bar g \bar p) \leq R/2$,
  there exists a unique preimage $g\in G$ of $\bar g$ such that  $d_X(p, g p) \leq  R/2$.
  \end{cor}

The following corollary says that we have uniform properness for the action of deep enough Dehn Fillings of $G$ on the corresponding Dehn filled space.

  \begin{cor}\label{cor_uniform_properness}
    Let $K_i$ be a sequence of Dehn kernels of $G$ of depth $R_i$ where $R_i$ goes to infinity.
    Let $\bar X_i=\dot X_{R_i}/K_i$ be the corresponding Dehn filled space.
    Then given any $C,D$, there exists $M$ such that for all $i$ large enough and all $\bar x\in \bar X_i\setminus \bar\calh_{D}^{\bar X_i}$,
    there are at most $M$ elements $\bar g\in \bar G_i$ such that  $d_{\bar X_i}(\bar x,\bar g\bar x)\leq C$.    
  \end{cor}

  \begin{proof}
    Consider  $Y=X\setminus \calh_D$. Since $Y$ is cocompact and $G$ acts properly on $X$, there exists a bound $M$ such that 
for all $y\in Y$, there are at most $M$ elements $g\in G$ such that $d_X(x,gx)\leq C$.
Let $R_0$ be such that the $2C$-neighbourhood of $Y$ is contained in $X\setminus \calh_{R_0}$ (this exists by Lemma \ref{lem_eqv_horo}(\ref{it_horo_sep})).
Consider $i$ such that $R_i>R_0$, $\bar x\in \bar X_i\setminus
\bar\calh_{D}^{\bar X_i}$, 
 and $x\in Y$ such that $\pi_{K_i}\circ i_{R_i}(x)=\bar x$.
Then Lemma \ref{lem;lift} applies with $r=2C$, and says that for any $\bar g\in G/K_i$ such that $d_{\bar X_i}(\bar g\bar x,\bar x)\leq C$
there exists a unique preimage $g\in G$ such that $d_X(x,gx)\leq C$.
This proves the Corollary.
  \end{proof}

\begin{lemma}\label{lem_preserv_non_conj}  
 If  $G/K$  is a  deep enough Dehn filling of $(G,\calp)$
 (\ie a proper Dehn filling by a Dehn kernel of
 sufficiently high depth), and if $P,Q$  in
 $\calp$ are  infinite and not conjugate in $G$, then their images $\bar P$ and $\bar Q$  are non-conjugate in $G/K$. 
\end{lemma}

\begin{proof} 
By Lemma \ref{lem_barsep}, if $\bar G$ is a Dehn filling of depth $R$,
for any peripheral group $P'$, its image $\bar P'$ is the stabilizer
of  $\bar B_{P'}\in \bar \calh_{0}^{\bar X_R}$,  
the  image of the trace
  of the horoball associated to $P'$. 
Let $\delta_1$ be the constant given by  Theorem \ref{theo;VRF} (which
can serve as a hyperbolicity constant for all  Dehn filled spaces
produced by the theorem). 
Consider $Y=X\setminus \calh_{0}$ the complement of the initial system of horoballs,
and $D\subset X$ a compact set such that $GD\supset Y$.

By uniform properness of the action, there exists $M$ and $R$ as follows. Consider a Dehn kernel $K$ of depth $\geq R$,
 $\bar X_R=\dot X_R/K$ the Dehn filled space. 
Then for any  $x\in \bar
 X\setminus \bar\calh_0 ^{\bar X_R}$,  
there are at most $M$
elements of $\bar G$ that move $x$ by at most $100\delta_1$.

 If the depth of $K$ is large enough, 
the image $\bar P$ of
$P$ in $\bar G=G/K$ has cardinality at least $M+1$. 

Then we claim that $\bar P$ is not conjugate in $\bar G$ to the image of any
 peripheral subgroup $Q\in\calp$ not conjugate to $P$ in $G$.
 If it was, then up to changing $Q$ to a conjugate,  $\bar P$ would fix two distinct traces of
horoballs $\bar B_P,\bar B_Q\in  \bar \calh_0^{\bar X_R}$.  
Let $\gamma$ be a geodesic joining the apex of $\bar B_P$ to the apex of $\bar B_Q$.
Then all elements of $\bar G$ move by at most $20\delta_1$ all the points on $\gamma$.
Since $\gamma$ intersects $\bar X_R\setminus \bar\calh_0^{\bar X_R}$, 
any point $x$ in the intersection is moved by at most $100\delta_1$
by $M+1$ elements, a contradiction.
\end{proof}

\subsection{Cofinality}

\begin{dfn}[Cofinal sequence of subgroups]\label{dfn_cof}
A sequence  (or a set) of subgroups $H_i$ of a group $G$ is \emph{cofinal} if for each finite subset $A\subset G\setminus \{1\}$,
one has $H_i\cap A=\es$  for all but finitely many $i$.
\end{dfn}

\begin{lemma}\label{lem;cof}  
  For each $j\in\{1,\dots,k\}$, consider $(N^i_j)_{i\in\bbN}$  a sequence of normal subgroups  $N^i_j\normal P_j$,
and consider $K_i=\ngrp{N^i_1,\dots,N^i_k}_G\normal G$.

Then $(K_i)_{i\in\bbN}$ is cofinal in $G$ if and only if for all $j$, $(N^i_j)_{i\in\bbN}$ is cofinal in $P_j$.

Moreover, in this case, 
$K_i$ is a proper Dehn kernel for all $i$ large enough
and
$$\lim_{i\ra \infty} \depth(K_i)=+\infty$$
(see Definition \ref{dfn_depth}).
  \end{lemma}

\begin{proof}   
If the sequence $K_i$ is cofinal, then so is $N^i_j$ since $N^i_j\subset K_i$.

Assume conversely that $N^i_j\normal P_j$ is cofinal. 
Fix $R\geq 0$, and consider the finite set $S_R\subset G\setminus\{1\}$ given by Theorem \ref{theo;VRF}.
Since $N^i_j$ is cofinal in $P_j$, 
it eventually avoids $S_R$. 
This shows that $\depth(N^i_j)\ra \infty$, and the second assertion
follows. 
To prove that $K_i$ is cofinal, 
consider $g\in G\setminus\{1\}$, and $p\in X$ the base point. 
Take $R= 10 d_X(p,gp)$,  
and $i$ large enough so that $\depth(N^i_j)\geq R$ for all $j$. 
We apply Corollary \ref{cor_lift}: 
if the image $\ol g$ of $g$ in $G/K_i$ is trivial, then
$g$ and $1$ are two elements of $G$ moving $p$ by at most $R/2$,
and which map to $\ol g \in G/K_i$.
The uniqueness statement in Corollary \ref{cor_lift} says that $g=1$, a contradiction.
\end{proof}

\section{Asymptotic type preservation}

\subsection{Peripheral structures and type preservation}\label{sec_type}

A \emph{peripheral structure} $\calq$ of a group $G$, is a family of subgroups stable under conjugation.
We say that $H<G$ is \emph{sub-$\calq$} if there exists $P\in\calq$ such that $H\subset P$.
For example, if $(G,\calp)$ is a relatively hyperbolic group, $H$ is sub-$\calp$ if and only if it is parabolic.

Given a relatively hyperbolic group $(G,\calp)$, it will be convenient to define
$\PF$ the union of $\calp$ together with all non-parabolic finite groups.
Thus, a subgroup $H<G$ is sub-$\PF$ if and only if it is finite or
parabolic. By 
Lemma \ref{lem_loc_parab}, this happens if and only if $H$ contains no
loxodromic element. 

If $G/K$ is a proper Dehn filling of $G$ and $\bPF$ is the image of $\PF$ in $G/K$, then 
a subgroup $H<G/K$ is sub-$\bPF$ if it is contained in the image of a parabolic or finite subgroup of $G$.
By construction such a group $H$ is either finite, or fixes an apex in
every Dehn filled space $\bar X_R$.

We say that $(G,\calq)\simeq (G',\calq')$ if there is an isomorphism $\phi:G\ra G'$ sending $\calq$ to $\calq'$.
It will be convenient to use the following slightly weaker condition.

\begin{dfn}\label{dfn_type}
  Let $(G,\calq),(G',\calq')$ be two groups with peripheral structures.
  We say that a monomorphism $\phi:G\to G'$ is \emph{type-preserving} if 
  for each subgroup $H<G$, $H$ is sub-$\calq$ if and only if $\phi(H)$ is sub-$\calq'$.
\end{dfn}

\begin{rem}\label{rk_type}
Consider $(G,\calp),(G',\calp')$ two relatively hyperbolic groups, and let $\calp_\infty\subset \calp$ (resp $\calp'_\infty\subset \calp'$) 
be the set of groups of $\calp$ (resp $\calp'$) that are infinite.
  Then  $\phi:(G,\PF)\ra(G',\PF')$ is  a type preserving isomorphism if and only if $\phi(\calp_\infty)=\calp'_\infty$.
However, it does not imply in general that $\phi(\calp)=\calp'$, as
illustrated by the case where $G=G'$, $\phi=\id$, with $\calp'=\calp\cup [F]$ where
   $F$  is a finite subgroup of a group in $\calP$ but not itself in
   the collection $\calp$. 
\end{rem}

\subsection{Asymptotic type preservation of Dehn fillings} \label{sec_type_pres}

  \begin{lemma}[Preimages of finite groups]\label{lem;elli}
    Let $(G,\calp)$ be a relatively hyperbolic group.
     There exists $R_0$,  and a finite collection   $\calC$ of finite subgroups of $G$,   such that  
     if $K$ is a proper Dehn kernel with $\depth(K)\geq R_0$ of $G$, 
     the following holds. 

     If $F< \bar G$ is a finite subgroup, either $F$  
    is contained in the image of a
     parabolic subgroup of $G$, or $F$ is the image of a conjugate of a group in
     $\calC$. 
  \end{lemma}

In other words, denoting by $\bPF$ the image of $\PF$ in $G/K$, the lemma says that if $K$ is deep enough,
then every finite subgroup of $G/K$ is sub-$\bPF$.

  \begin{proof}
Let $\delta_1$ be the hyperbolicity constant of the spaces $\bar X_R=\dot X_R/K$ 
(independent of $R$ and $K$) as in Theorem \ref{theo;VRF}.
Consider $Y=X\setminus \calh_{100\delta_1}$ the complement of the system of horoballs for $R=100\delta_1$.
For each $y\in Y$, consider the set of finite subgroups of $G$ that move $y$ by at most $5\delta_1$.
Since the action of $G$ on $Y$ is proper and cocompact, the set $\calc$ of all finite subgroups thus obtained as $y$ varies in $Y$ is a finite set of conjugacy classes of finite subgroups.

Let $R_0=200\delta_1+D_{\ref{lem_eqv_horo}}$, where $D_{\ref{lem_eqv_horo}}$ is the constant appearing in Lemma \ref{lem_eqv_horo} (\ref{it_horo_sep})
 so that for all $x\in Y$, $B_X(x,100\delta_1)\cap\calh_{R_0}=\es$.
Let $K$ be a proper Dehn kernel with $\depth(K)\geq R_0$, and $F<G/K$ a finite subgroup.
Since by Theorem \ref{theo;VRF},   $\bar X = \bar X_{R_0} = \dot{X}_{R_0}/K$ is $\delta_1$-hyperbolic, 
there exists a point $\bar x\in \bar{X}$, such that, for all
$f\in F$,  $d_{\bar X}(f\bar x, \bar x)\leq 5\delta_1$  (see \cite[Lemma 3.3 p460]{BH_metric}).

First assume that $\bar x\in \pi_K\circ\iota_{R_0}(Y)$, and let $x\in Y$ be a preimage of $\bar x$. 
Since  
$B_X(x,100\delta_1) \cap \calh_{R_0}=\es$, 
Lemma \ref{lem;lift} implies that $\pi_K\circ \iota_{R_0}$ induces an isometry
$B_X(x,50\delta_1)\ra B_{\bar X}(\bar x,50\delta_1)$.
Denote by $j:B_{\bar X}(\bar x,50\delta_1)\ra B_X(x,50\delta_1)$ its inverse.
By Lemma \ref{lem;lift}, each $f\in F$ has a unique lift $\Tilde f\in G$
such that $d(x,\Tilde f x)\leq 50\delta_1$. If $f=f_1f_2$, then $\Tilde f_1\Tilde f_2$
is another such lift since $d_{X}(\tilde f_1\tilde f_2 x,x)\leq 5\delta_1+5\delta_1 \leq 50\delta_1$.
It follows that the assignment $f\mapsto \Tilde f$ is a morphism, and its image is a group in $\calc$.
This is the desired lift of $F$.

Assume now that $\bar x\notin \pi_K\circ\iota_{R_0}(Y)$.
Then $\bar x$ lies in some trace of $\calh_{100\delta_1}$-horoball
$\bar B\in \bar \calh_{100\delta_1}^{\bar X_{R_0}}$.    
By Lemma \ref{lem_barsep}, $\bar \calh_{100\delta_1}^{\bar X_{R_0}}$ is $100\delta_1$-separated, so
$F$ preserves $\bar B$, and $F$ is contained in the image in $\bar G$ of a  peripheral group $P\in\calp$.
\end{proof}

\begin{lemma}[Preimage of non-loxodromic elements] \label{lem;pre_nonlox}
    Let $(G,\calp)$ be a relatively hyperbolic group.
  There exists $R_0$   such that  if $K$ is a proper Dehn kernel  of $G$ with $\depth (K)\geq R\geq R_0$, 
 and if  $\bar g$ acts as a non loxodromic isometry on
 $\bar{X}_R=\dot X_R/K$, then $\bar g$ has finite order or has a
 preimage $g\in G$ that is  
conjugate into some $P_i$. 
\end{lemma}

In other words, the lemma says that any element $\bar g\in G/K$ that is not loxodromic is sub-$\bPF$.
In particular, it is elliptic in $\bar X$.

\begin{proof}
Let $Y=X\setminus \calh_{0}$ be the complement of the system of
horospheres corresponding to $R=0$.  
Let $\delta_1$ be the hyperbolicity constant (independent of $R$ and
$K$) of the spaces $\bar X_R=\dot X_R/K$ as in Theorem \ref{theo;VRF}.
Let $R_0=100\delta_1+D_{\ref{lem_eqv_horo}}$  (where $D_{\ref{lem_eqv_horo}}$
is the constant appearing in Lemma \ref{lem_eqv_horo} (\ref{it_horo_sep}))
so that for all $x\in Y$ and any $R\geq R_0$, $B_X(x,100\delta_1)\cap\calh_R=\es$.

Under the assumption of the lemma, for all $n\in \bbN$, there exists
$\bar x \in \bar X_R$  such that $d(\bar g^j \bar x, \bar x) <20\delta_1$ for all $j\leq n$.
Since $G$ acts properly and cocompactly on $Y$, there exists $n$ such
that, for all $x\in Y$,   $\#\{g|d_X(x,gx)\leq 100\delta_1\}\leq n$.

We argue as in Lemma \ref{lem;elli}.
If $\bar x\in \pi_K\circ\iota_R(Y)$, then consider $x\in Y$ a preimage of $\bar x$.
Since $R\geq R_0$, $B_X(x,100\delta_1)\cap \calh_{R}=\es$,
so we can apply Lemma \ref{lem;lift} saying that 
$\pi_K\circ \iota_{R}$ induces an isometry
$B_X(x,50\delta_1)\ra B_{\bar X_R}(\bar x,50\delta_1)$,
and that for all $j\leq n$, $\bar g^j$ has a unique preimage $g_j\in G$ such that
$d(x,g_j x)\leq 50\delta_1$. 
The uniqueness guarantees that $g_{j+1}=g_1g_j$ for all $j< n$
since $d(g_{1}g_jx,x)\leq 20\delta_1+20\delta_1\leq 50\delta_1$.
Thus $g_j=g_1^j$, and our choice of $n$ guaranteed that there exists $j_1<j_2\leq n$
 $g_1^{j_1}=g_1^{j_2}$. Thus, $g_1$ has finite order, and so has $\bar g$.

Assume now that $\bar x\notin \pi_K\circ\iota_{R}(Y)$.
Then as in Lemma \ref{lem;elli}, 
$\bar x$ lies in  some trace of horoball $\bar B\in \bar
\calh_{100\delta_1}^{\bar X_{R}}$. Since by Lemma 
\ref{lem_barsep}
$\bar \calh_{100\delta_1}^{\bar X_{R}}$  
 is $100\delta_1$-separated, 
$\bar g$ preserves $\bar B$, so again by Lemma \ref{lem_barsep} $\bar
g$ is contained in the image of a peripheral group $P\in\calp$.
\end{proof}

The following lemma is the basis of asymptotic type preservation for Dehn fillings.

\begin{lemma} \label{lem;cof2}
  Let $(G,\calp)$ be relatively hyperbolic, and $g\in G$ a  hyperbolic element.

  Then there exists $R_g\geq 0$ (depending on $g$)
  such that for any proper Dehn kernel $K$ with $\depth(K)\geq R_g$, 
  the image $\bar g$ of $g$ in $G/K$ is not sub-$\bPF$: it has infinite order, and is not contained in the image of a parabolic group of $G$.

  Moreover $\bar g$ acts loxodromically on $\bar X_R=\dot X_{R}/K$.
\end{lemma}

\begin{cor}[Asymptotic type preservation]\label{cor_asymptotic}
  Let $(G,\calp)$ be relatively hyperbolic, and $H<G$ be a subgroup.

  \begin{enumerate}\item \label{it_as1}
    If $H$ is sub-$\PF$, then its image in $G/K$ is sub-$\bPF$ for
    every Dehn kernel $K$.

  \item \label{it_as2}
    If $H$ is not sub-$\PF$, then
    there exists $R_H$ (depending on $H$) such that for any proper Dehn kernel
    $K$ of depth $\geq R_H$, the image $\bar H$ of $H$ in $G/K$ is not sub-$\bPF$: it
    is infinite, and not contained in the image of a parabolic group
    of $G$.

  \end{enumerate}
\end{cor}

\begin{proof}[Proof of the corollary]
The first assertion follows immediately from the definition of $\bPF$.
So assume that $H$ is not sub-$\PF$, \ie that $H$ is infinite and not parabolic.
Then by Lemma \ref{lem_loc_parab}, $H$ contains a loxodromic element $g$.
Applying Lemma \ref{lem;cof2}, we get that for deep enough
proper Dehn kernels, $\bar g$ is not sub-$\bPF$, hence neither is $\bar H$.
\end{proof}

\begin{proof}[Proof of Lemma \ref{lem;cof2}]
We first recall that if an isometry $g$ of a $\delta$-hyperbolic space $X$ is loxodromic, then 
for any $x\in X$ there exists $M\geq 0$ and $n\in\bbN$ such that $x,g^nx,g^{2n}x$ are 
are at arbitrarily large distance from each other, and are
$M$-almost aligned: $$d(x,g^{2n}x) \geq d(x,g^nx) + d(g^nx,g^{2n}x) -M.$$
Indeed, by definition, $\{g^nx|n\in\bbZ\}$ is a quasi-geodesic, so for all $n$, $g^n x$ is at bounded distance
from $[x,g^{2n}x]$, which proves our claim.
Conversely, in a $\delta$-hyperbolic space,  
if there exists a point $x$, and $n$ such that $x,g^nx,g^{2n}x$ are $M$-almost aligned (in the above sense)
and $d(x,g^n x)>M+2\delta$, then $d(x,g^{2n} x)>d(x,gx)+2\delta$
so $g$ is loxodromic by \cite[Lemme 9.2.2]{CDP}.

Under the hypotheses of the lemma, let $p\in X$ be the  base point,
and $M$ be such that $p,g^np,g^{2n}p$ are $M$-almost aligned for all $n$.
 Let $\delta_1$ be the universal hyperbolicity constant 
given by Theorem \ref{theo;VRF}.
Let $n$ be such that $d(p,g^n p)>M+2\delta_1$, and $R=4d(p,g^np)$.
 Consider $K$ is a proper Dehn kernel of depth $\geq R$,
and the $\delta_1$-hyperbolic space $\bar X_R=\dot X_R/K$.

Denote by $\ol p$ the image of the base point in $\bar X_R$ under $\pi_{K}\circ\iota_R$.
Since by Corollary \ref{cor_lift}, $\pi_{K}\circ\iota_R$ is isometric on $B_X(p,R/2)$ and $R/2\geq 2 d(p,g^np)$,
we get that $\ol p,\bar g^n\ol p,\bar g^{2n}\ol p$ are $M$-almost aligned in $\bar X_R$,
and $d_{\bar X_R}(\bar p,\bar g \bar p)=d_{X}(p,gp)>M+2\delta_1$.
We conclude that $\bar g$ acts as a loxodromic isometry of $\bar X_R$. 
This implies that $\bar g$ is not  sub-$\bPF$ since sub-$\bPF$ elements have finite order or fix a point in $\bar X_R$.
\end{proof}

Similarly to lemma \ref{lem;cof2}, we have:

\begin{lem}\label{lem_F2}
  Let $h_1,h_2\in G$ be two hyperbolic elements (for their action on
  $X$) whose four 
  fixed points at infinity are distinct.
Then there exists $N$ and $R_0$ such that for each Dehn kernel $K$ of depth at least $R_0$, $\grp{h_1^N,h_2^N}$ is free and embeds
in $G/K$.
\end{lem}

\begin{proof}
  Let $p\in X$ be a base point, and denote by $\omega_i^+$ and $\omega_i^-$ the attractive and repelling fix points of $h_i$ at infinity.
Since the four points $\omega_1^+,\omega_1^-,\omega_2^+,\omega_2^-$ are distinct, there exists a bound $M$ such that 
for all $i,j\in\{1,2\}$ and $\eps,\eps'\in\{\pm1\}$, and all $N\geq 1$,
the Gromov product $(h_i^{\eps N} p| h_j^{\eps' N} p)_p$ is bounded by $M$ unless $(i,\eps)=(j,\eps')$

Choose $N$ so that $d(p,h_1^Np)$ and $d(p,h_2^Np)$ are larger that $M'$, with $M'$ very large compared to $M$ and to $\delta_1$ (the common hyperbolicity constants
of the Dehn filled spaces) so that any $M'$-local $(1,M)$-quasigeodesic of length $\geq M'$ in a $\delta_1$-hyperbolic space has distinct endpoints.
Now choose $R_0$ so that $B(p,2M')$ is disjoint from $\calh_{R_0}$.
Let $K$ a Dehn kernel of depth $\geq R_0$, and consider $\bar X_R=\dot X_R/K$ the corresponding Dehn filled space.
Let $\ol p$ be the image of $p$ in $\bar X_R$. Then the ball of radius $M'$ around $p$ in $X$ is isometric to the ball of radius $M'$ around $\ol p$ in $\bar X_R$.

Now consider $w=h_{i_1}^{\eps_1}h_{i_2}^{\eps_2}\dots h_{i_n}^{\eps_n}$ is a reduced word in $h_1^{\pm 1},h_2^{\pm 1}$.
Then the concatenation of the segments 
$$[\bar p,h_{i_1}^{\eps_1}\bar p],\ h_{i_1}^{\eps_1}[\bar p,h_{i_2}^{\eps_2}\bar p],\ \dots,\  h_{i_1}^{\eps_1}\dots h_{i_{n-1}}^{\eps_{n-1}}[\bar p,h_{i_n}^{\eps_n}\bar p]$$
is a $M'$-local $(1,M)$-quasigeodesic in $\bar X_R$.
By choice of $M'$, $w \bar p\neq \bar p$, so $w\neq 1$, and $\grp{h_1,h_2}$ embeds in $G/K$.
\end{proof}

\section{Rigidity theorems}

The goal of this section is the following rigidity result that shows that non-elementary rigid relatively hyperbolic groups are determined
by their Dehn fillings.
This is the crux of the paper.

We first introduce the notion of $\Zmax$-rigidity, which is slightly less restrictive than the rigidity condition stated in the introduction:
if  $(G,\calp)$ is rigid, then it is $\Zmax$-rigid.

Recall that an (infinite) virtually cyclic group either maps onto $\bbZ$ with finite kernel,  
or maps   onto the infinite dihedral group with finite kernel. 
It has infinite center if and only if it maps onto $\mathbb{Z}$.    
We say that a subgroup $H$ of a group $G$ is \emph{a $\Zmax$ subgroup}
if $H$ is not parabolic, but is virtually cyclic with infinite center, 
and is maximal for these properties (with respect to inclusion). 

\begin{dfn}\label{dfn_Zmax_rigid}
We say that a relatively hyperbolic group  $(G,\{[P_1],\dots,[P_n]\})$ is \emph{$\Zmax$-rigid} if 
it has no non-trivial splitting (as an amalgamation, or an HNN extension) 
over a finite, a parabolic, or a $\Zmax$ subgroup 
such that each $P_i$ is conjugate in some factor of the amalgamation (or of the HNN extension).   
\end{dfn}

Recall that we denote by $\PF$ (resp $\PF'$) the family of
peripheral groups in $\calp$ together with
finite non-parabolic subgroups of $G$ (resp $G'$).
If $K_i$ (resp $K'_i$) is a cofinal sequence of Dehn kernels, we denote by
$\bPF_i$ (resp $\bPF'_i$) the image $\PF$ (resp $\PF'$) in $G/K_i$ (resp $G'/K'_i$), see Section \ref{sec_type_pres}.

\begin{thm}\label{thm_DF_carac}
Consider   $(G,\calp)$,  $(G',\calp')$ two relatively
  hyperbolic groups. Assume that both are non-elementary and
  $\Zmax$-rigid.

Consider cofinal sequences of Dehn kernels $K_i\normal G$, $K'_i\normal G'$ and
assume that there are type preserving isomorphisms  
$\phi_i: (G/K_i,\bPF_i)\stackrel{\sim}{\to}(G'/K'_i,\bPF'_i)$.

Then there is a type preserving isomorphism $\phi:(G,\PF)\xra{{}_\sim} (G',\PF')$ that induces infinitely many
of the $\phi_i$, up to composition with  inner automorphisms.
\end{thm}

  The conclusion means that there exists $h_i\in\bGp_i$ 
such that, for infinitely many indices $i$, 
the following diagrams commutes
$$\xymatrix@C=2cm{G\ar@{.>}[r]^{\phi} \ar@{->>}[d]^{p_i} & G' \ar@{->>}[d]^{q'_i}
\\ \DFi \ar@{->}[r]^{\ad_{h_i}\circ\phi_i} & \DFpi }$$
where $\ad_{h_i}:\DFpi\ra \DFpi$ is the inner automorphism $x\mapsto h_i x h_i\m$.

\begin{rem} Note that the assumption on the sequence of Dehn kernels
  considers the peripheral structures $\bPF_i, \bPF'_i$ and not
  $\overline{\calP}_i, \overline{\calP}'_i$. This is because the
  asymptotic type preservation holds for the former structures, not
  the later.

By Remark \ref{rk_type}, if no group in $\calp\cup\calp'$ is finite, then the fact that $\phi :(G,\PF)\stackrel{\sim}{\to} (G',\PF')$ is type preserving
means that $\phi(\calp)=\calp'$.
The assumption that $\phi_i: (G/K_i,\bPF_i)\stackrel{\sim}{\to}(G'/K'_i,\bPF'_i)$ holds in particular if 
$\phi_i$ sends $\bar\calp_i$ to $\bar\calp'_i$. Thus the formulation
of Theorem \ref{thm_carac_intro} in the introduction is a particular
case of Theorem \ref{thm_DF_carac}.
\end{rem}

We will prove this theorem in several steps. 
Given isomorphisms $G/K_i\ra G'/K'_i$,
we will either  produce an action of $G$ on an $\bbR$-tree from the induced action of $G$
on the Dehn filled spaces for $G'/K'_i$,
or  produce a type-preserving monomorphism $\phi:G\ra G'$
 commuting with $\phi_i$ up to inner automorphisms for infinitely many indices $i$.

 By a symmetric argument, if each $\phi_i$ is an isomorphism, then there
 is a type-preserving monomorphism $\psi:G'\ra G$. 
 If $\phi$ or $\psi$ is not onto,
 then $\psi\circ\phi$
 is a type-preserving monomorphism $G\ra G$ that is not onto.
 We then prove that this implies that $G$ has an action on $\bbR$-tree.
The proof of Theorem \ref{thm_DF_carac} will be given in subsection \ref{sec_DF_carac}.

\subsection{Morphisms with bounded displacement}\label{sec_bounded}

It is convenient to gather the setting in which we will be
working.   
\begin{notations}\label{notations}
Let $(G,\calp)$ be relatively hyperbolic, and $X$ be a $\delta_c$-hyperbolic graph as in Definition 
\ref{def;RHG}.
For all $j$, consider $P_j^i\normal P_j$ a cofinal sequence of normal subgroups,
and $(K_i)_{i\geq 0}$ the corresponding cofinal sequence of Dehn kernels.
Let $i_0$ be such that $K_i$ is a proper Dehn kernel for $i>i_0$.
For $i\geq i_0$, we denote by $R_i=\depth(K_i)$ (or $R_i=i$ if $\depth(K_i)=\infty$). 

We denote the Dehn filled group by $\bG_i=G/K_i$ with 
$q_i:G\ra \bG_i$ the quotient map.  
We denote  the Dehn filled space by $\bar{X}_i=\dot{X}_{R_i}/K_i$ where $\dot X_{R_i}$ is the cone-off at depth $R_i$ (see Proposition \ref{prop_coneoff} and Theorem
\ref{theo;VRF}).
We also denote by $\bar\calp_i$ (resp $\bPF_i$) the image of $\calp$ (resp $\bPF$) in $G/K_i$.
\end{notations}

Given a cofinal sequence $K'_i$ of Dehn kernels in the relatively hyperbolic group $(G',\calp')$, 
we use the corresponding notations $X',R'_i,\dot X'_{R'_i},\bar X_i'$ etc.

If   $A$ is a finite set of isometries of a metric space $X$, we define
its least displacement 
 by  $$\displ_{ X}(A)=\inf_{x\in X} \max_{a\in A} d_X(x,ax).$$ 

Given $\phi_i:\DFi\ra \DFpi$, the morphism $\phi_i\circ q_i$ 
gives an action of $G$ on $\bar X'_i$ by precomposition.
To measure its displacement, fix $S$ a finite generating set of $G$,
 and define $$||\phi_i||_{\bar X'_i}=\displ_{\bar X'_i} ( \phi_i\circ q_i(S)).$$

  \begin{prop}\label{prop;mono}
    Let  $(G,\calP)$,  $(G',\calP')$ be two 
    non-elementary 
    relatively hyperbolic groups.
    Let $K_i\normal G$, $K'_i\normal G'$ be cofinal sequences of 
    Dehn kernels of $(G, \calP), (G', \calP')$.

    Assume that there exists $M>0$, and for each $i$,
    a type preserving monomorphism $\phi_i: (\DFi, \bPF_i) \to (\DFpi,\bPF'_i)$, 
    such that $||\phi_i||_{\bar X'_i}\leq M$.

Then there exist a type-preserving monomorphism $\phi :(G,\PF)\ra (G',\PF')$
and inner automorphisms $\ad_{h_i}$ of $\bG_i$
making the following diagrams commute 
for infinitely many indices $i$:
$$\xymatrix@R=8mm@C=15mm@M=2mm{G\ar@{^{(}-->}[r]^{\phi} \ar@{->>}[d]^(0.4){q_i} & G' \ar@{->>}[d]^(0.4){q'_i}
\\ \DFi \ar@{^{(}->}[r]^{\ad_{h_i}\circ\phi_i} & \DFpi }$$
  \end{prop}

\begin{rem}
  Even if we assume that $\phi_i$ are isomorphisms,  we cannot guarantee that $\phi$ is an isomorphism. 
Indeed, Bridson and Grunewald produced in \cite{BriGru_Grothendieck}  examples of non-isomorphic residually finite finitely presented groups $A, B$, 
with an injective morphism $A\hookrightarrow  B$ inducing an isomorphism at the level of pro-finite completions 
(\ie isomorphisms at the level of every characteristic finite quotient).  
Then, consider the relatively hyperbolic groups $A*A$, and $B*B$, and let $A_i\subset A$ and $B_i\subset B$ be the intersection 
of all subgroups of index at most $i$. Since the inclusion $A\hookrightarrow B$ induces  isomorphisms
$A/A_i\xra{\sim} B/B_i$ for all $i$, it induces isomorphisms $\phi_i:A/A_i*A/A_i\xra{\sim} 
B/B_i*B/B_i$ with bounded displacement (for the actions on the Bass-Serre trees).
However, the limiting morphism $\phi:A*A\hookrightarrow B*B$ is not onto.

\end{rem}

\begin{rem}\label{rem;wo_type_preservation_inDF}
 One can remove the assumption that all $\phi_i$ are type preserving, and replace it by the assumption that 
for $i$ large enough, the groups in  $\bar\calP_i, \bar\calP'_i$ are
small 
 (this holds in particular if the groups in $\calp,\calp'$ themselves are small,  or if we consider only 
finite Dehn fillings).
Under this assumption, the same conclusion holds 
except that one looses the property that the monomorphism $\phi$ is type preserving. We indicate in the course of the proof, the modification of the argument for this variation. 
\end{rem}

\begin{proof}
Let $\bar x'_i\in \bar X'_i$ be a point moved by at most $2M$ by all elements of the generating set $S$.
Informally, we first prove that $\bar x'_i$ stays in the thick part.

Since $G$ is non-elementary, 
$q_i(G)$ is 
not sub-$\bPF_i$ for $i$ large enough by Corollary \ref{cor_asymptotic}. 
Since  $\phi_i$ is type preserving, the group  $\phi_i(\bG_i)$ is therefore not
 sub-$\bPF'_i$ for $i$ large enough. 

In the context of  Remark \ref{rem;wo_type_preservation_inDF}, 
we still have that  $\phi_i(\bG_i)$ is not   sub-$\bPF'_i$ for $i$ large enough, since $\bG_i$ contains non-abelian free groups 
by Lemma \ref{lem_F2}.

Consider $i$  large enough so that $\depth(K'_i)\geq 2M$.
Consider the system of horoballs $\calh'_{2M}$ of $X'$,
and $\bar \calh^{\bar X'_i}_{2M}$ its trace on $\bar X'_i$.  
By Lemma \ref{lem_barsep}, $\bar \calh^{\bar X'_i}_{2M}$ is
$2M$-separated. 
We claim that $\bar x'_i$ lies in the complement of $\bar \calh^{\bar X'_i}_{2M}$. 
Assume on the contrary that $\bar x'_i\in \bar B$ for some $\bar B\in \bar \calh^{\bar X'_i}_{2M}$, and argue towards a contradiction. 
Since $||\phi||_{\bar X'_i}\leq M$, and $\bar \calh^{\bar X'_i}_{2M}$ is $2M$-separated,  
all generators of $G$ (more precisely, all elements of $\phi_i\circ q_i(S)$) preserve $\bar B$, and so does $\phi_i\circ q_i(G)$.
Since by Lemma \ref{lem_barsep}, the stabilizer of $\bar B$ is the image $\bar P'$ of a peripheral subgroup $P'\in  \calp'$, 
this contradicts that $\phi_i(\bar G_i)$ is not sub-$\bPF'_i$, and our claim is proved.

Consider $D\geq \diam (X'\setminus \calh'_{2M})/G$, so that
$\bar X_i\setminus \bar \calh^{\bar X'_i}_{2M}$ is contained in the $D$-neighbourhood of the orbit of $\bar p'_i$, the image of the base point in $\bar X'_i$ 
(Lemma \ref{lem_barsep}). 
Then up to post-composing $\phi_i$ by inner automorphisms of $\DFi$  and changing $\bar x'_i$ in its orbit accordingly, 
we can assume that $d(\bar x'_i,\bar p'_i)\leq D$.
In particular, for all $s\in S$, $d(\bar p'_i,\phi_i\circ q_i(s)\bar p'_i)\leq C$ where $C=2D+2M$.
It follows that 
for each $g\in G$ of word length $|g|_{S}$, one has $d(\bar p'_i,\phi_i\circ
q_i(g)\bar p'_i)\leq C|g|_{S}$.

Now we assign to each $g\in G$ a finite set $F_{g}\subset G'$ and a sequence $g'_i\in F_g$ of candidates for $\phi(g)$.
 Given $g\in G$, consider $i$ large enough so that  the depth $R'_i$ is at least  $R'_i\geq 2C|g|_{S}$.
By Corollary \ref{cor_lift}, there exists a unique $g'_i\in G'$ such that $q'_i(g'_i)=\phi_i\circ q_i(g)$
and $d_X(p',g'_i p')\leq C|g|_{S}$. By properness of the action of $G$ on $X$,
the set of elements $g'_i\in G$ varies in a finite set $F_{g}$ as $i$ varies.
Moreover, if $g=uv$ for some $u,v\in G$, 
and if ${\rm ker} (q_i)$ is of depth at least  $2C(|u|_{S}+|v|_{S})$, then 
the lifts $u'_i,v'_i\in G'$ of $u,v$ satisfy $g'_i=u'_iv'_i$ by uniqueness of the lift of $uv$ that moves $p'$
by at most $C(|u|_{S}+|v|_{S})$.

Now one produces $\phi:G'\ra G$ by selecting $\phi(g')$ among the finitely many elements $\{g_i\}_{i\in\bbN}$,
in a consistent way.
This can be done by extracting subsequences and using a diagonal argument.
It is easier to define $\phi$ using  a non-principal ultrafilter
$\omega$ on $\bbN$. 
Given such an ultrafilter, any sequence $(x_i)_{i\in \bbN}$ of elements of a finite set $F$
defines uniquely an element $\lim_\omega x_i\in F$, characterized as the unique element $x\in F$ such that $\omega(\{i\in \bbN |x_i=x\})=1$
(viewing $\omega$ as a finitely additive measure on $\bbN$ with values in $\{0,1\}$).
In this language, we define $\phi(g)=\lim_\omega g'_i\in F_{g}$ where $g'_i\in F_{g}$ is the sequence of elements defined above.
Since for all $g,u,v\in G'$, and for all $i$ large enough $g'_i=u'_iv'_i$ with the notations above, 
$\phi$ is a morphism.
The two morphisms $\phi_i\circ q_i$ and $q'_i\circ\phi$ agree on the finite set $S$ for $\omega$-almost every $i$.
For each such $i$, the diagram commutes.

There remains to check that $\phi$ is one-to-one and, if the maps $\phi_i$ are type-preserving, that $\phi$ is also  
type-preserving.
Up to passing to a subsequence, we assume that the diagram commutes for all $i$.
If $g\neq 1$, then for $i$ large enough $q_i(g)\neq 1$ since $K_i$ is cofinal.
Since $\phi_i$ is injective, $\phi_i\circ q_i(g)\neq 1$.
Since the diagram commutes, $\phi(g)\neq 1$, and $\phi$ is injective.

Assume that $\phi_i$ is type preserving, and let's check that so is $\phi$.
If $H$ is sub-$\PF$ (\ie parabolic or finite), then $q_i(H)$ is sub-$\PF_i$ for all $i$ (Corollary  \ref{cor_asymptotic}(\ref{it_as1}))
so $\phi_i\circ q_i(H)$ sub-$\PF'_i$ since $\phi_i$ is type preserving.
On the other hand, if $\phi(H)$ was not sub-$\PF'$, then $q'_i(\phi(H))$ would not be 
sub-$\PF'_i$ for $i$ large enough by Corollary \ref{cor_asymptotic}(\ref{it_as2}), a contradiction.
Conversely if $H<G$ is not sub-$\PF$, one checks similarly that $\phi(H)$ is not sub-$\PF$ using Corollary \ref{cor_asymptotic}.
\end{proof}

\subsection{Morphisms between Dehn fillings with unbounded displacement}\label{sec_unbounded}

   In this section we assume 
that the displacement is unbounded, and we apply Bestvina-Paulin's argument to produce an action of $G$ on an $\bbR$-tree.
The crucial observation is that the hyperbolicity constant of $\bar X_i$ is independent of $i$.

   \begin{prop}\label{prop;R-tree}  
     Let $K_i, K'_i$ be cofinal sequences of 
     Dehn kernels of two relatively hyperbolic groups $(G, \calP), (G', \calP')$. Using notations \ref{notations},
     assume that for each $i$ there is a type preserving monomorphism 
     $\phi_i: (\DFi,\bPF_i) \hookrightarrow (\DFpi,\bPF_i')$ such that $||\phi_i||_{\bar X'_i}$ is unbounded.

     Then $G$ admits a non-trivial isometric 
     action on a $\mathbb{R}$-tree $T$ such that
     \begin{itemize}
     \item every peripheral  subgroup $P\in \calp$ fixes a point in $T$
       \item every arc stabilizer is finite, parabolic or $\Zmax$. 
     \end{itemize}
   \end{prop}

In this statement, an \emph{arc} in an $\bbR$-tree is a geodesic segment $[a,b]$ with $a\neq b$, and its
\emph{stabilizer} is its pointwise stabilizer, \ie the set of $g\in G$ fixing both $a$ and $b$.

We will prove this proposition shortly, but first, we give applications.    

We say that $G$  splits relative to $\calP$ if it acts without global fixed point on a simplicial tree $T$ 
and every group in $\calP$ is elliptic.
Using Bass-Serre theory \cite{Serre_arbres}, 
 this is equivalent to the fact that 
$G$ is isomorphic to the fundamental group of a non-trivial graph of groups, 
in which each group in $\calP$ is conjugate 
to a subgroup of a vertex group.
We say that a  splitting is \emph{over} a class $\calE$ of groups, 
if the edge stabilizers (or the edge groups of the graph of groups) belong to the class $\calE$.

We recall the following result of Rips theory. 
 
 \begin{theo}[\cite{GL6}, Theorem 9.14]\label{thm_rips}
Consider $(G,\calp)$  a relatively hyperbolic group where each peripheral group $P\in \calp$ is finitely generated. 
Assume that $G$ acts without global fix point on an $\bbR$-tree $T$ so that each $P\in\calp$ fixes a point in $T$,
and stabilizers of arcs in $T$ are  finite, parabolic or $\Zmax$. 

Then $G$ has a non-trivial action on a simplicial tree $S$, such that each $P\in \calp$ fixes a point in $S$, 
and edge stabilizers are finite, parabolic or $\Zmax$.
\end{theo}

   Applying Theorem \ref{thm_rips}, we immediately get a splitting as follows.

   \begin{cor}\label{cor;R-tree}
     Under the hypotheses of Proposition \ref{prop;R-tree}, 
     $G$ has  a non-trivial splitting relative to $\calp$,
     over a finite, parabolic or $\Zmax$ subgroup.
     \qed
   \end{cor}

We now prove  Proposition \ref{prop;R-tree}.

%%%

\begin{proof}

As above, we consider the action of $G$ on $\bar X'_i$ through the morphism $\phi_i\circ q_i$.
Recall that $\bar X'_i$ is $\delta_1$-hyperbolic for some constant $\delta_1$ independant of $i$.
Let $Y_i$ be the metric space obtained from  ${\bar{X}'_{i}}$ by rescaling the metric by the factor $1/\| \phi_i \|_{\bar{X'}_{i}}$.
Up to taking a subsequence, we can assume that this scaling factor goes to zero, and so does the hyperbolicity constant of $Y_i$.
Going to a limit, Bestvina and Paulin's argument \cite[2.6]{Pau_arboreal}
provides an action of $G$ on an $\mathbb{R}$-tree $T$. This action has no global fix point
since the minimal displacement of $S$ on $Y_i$ is $1$.

Moreover, up to extracting a subsequence, the actions $G\actson Y_i$ converge to $T$ in the equivariant Hausdorff-Gromov topology:
for any finite set $\{x_1,\dots, x_n\}\subset T$, any finite set $A\subset G'$, and any $\eps>0$, then for any sufficiently large $i$, %%%
there exists $x_1^{(i)},\dots, x_n^{(i)}\in Y_i$ such that for all $a\in A$, and all $j,k\in \{1,\dots,n\}$,
$|d_T(a x_j,x_k)-d_{Y_i}(a x^{(i)}_j,x^{(i)}_k)|\leq \eps$.

Using that $\phi_i$ is type preserving, we now prove that every
peripheral  subgroup $P\in\calp$ fixes a point in $T$.
If not, and since $P$ is finitely generated, 
there exists $g\in P$ that is hyperbolic in $T$ by Serre's Lemma %%%
\cite[1.14]{Sh_dendrology}.
Consider $x\in T$ in the axis of $g$, so that $d_T(x,g^2x)=2\times d_T(x,gx)= d_T(x,gx)+d_T(gx,g^2x)$, and these distances are all non-zero.
For $\eps <\frac{ d(x,gx)}{100}$, consider $i$ such that  there is $x^{(i)}$ an $\eps$-approximation point in $Y_i$.  %%%
%%%
Then $d_{Y_i}(x^{(i)},g^2x^{(i)})$ differs from $d_{Y_i}(x^{(i)},gx^{(i)})+d_{Y_i}(gx^{(i)},g^2x^{(i)})$ by at most $3\eps$. We may also choose $i$ such that the hyperbolicity constant of $Y_i$ is $<\eps$. Thus, in $Y_i$, the points $x^{(i)},gx^{(i)}, g^2x^{(i)}$ are almost aligned, and sufficiently far away to apply \cite[Lemme 9.2.2]{CDP} which ensures that $g$ acts as a loxodromic isometry on $Y_i$. This contradicts the assumption that  $\phi_i$ is type preserving.

We now study arc stabilizers in $T$. 
Let $[a,b]$ be a non-degenerate arc in $T$, and $H<G$ be its pointwise stabilizer. 
Assume that it is not finite, nor parabolic in $(G, \calP)$.  We are going to prove that $H$ is virtually cyclic with infinite center, and for that, we first look for an element that is loxodromic in some  $\bar X_i'$.

By Lemma \ref{lem_loc_parab}, there is a hyperbolic element $h\in H$.  By Lemma \ref{lem;cof2},  $q_i(h)$ is a loxodromic isometry on $\bar X_i$, for $i$ large enough, hence it is not sub-$\bPF_i$. By type preservation of $\phi_i$, its image $\phi_i\circ q_i(h)$ is not  sub-$\bPF'_i$.
Lemma \ref{lem;pre_nonlox} says that for $i$ large enough, any element of $\bar G'_i$ that is not sub-$\bPF'_i$ is a loxodromic isometry on $\bar X'_i$. 
Thus $h\in H$ is an element such that $\phi_i\circ q_i(h)$ is loxodromic in  $\bar X'_i$.
 In other words we are in the situation of the next lemma, and the proposition is proved.
\end{proof}

\begin{lemma}\label{lem;arcstab_VC}
   Let $K_i, K'_i$ be cofinal sequences of 
     Dehn kernels of two relatively hyperbolic groups $(G, \calP), (G', \calP')$. Using notations \ref{notations},
     assume that for each $i$ there is a  
 monomorphism 
     $\phi_i: (\DFi,\bPF_i) \hookrightarrow (\DFpi,\bPF_i')$ such that $||\phi_i||_{\bar X'_i}$ is unbounded.

Assume that $T$ is a limit tree of the spaces $Y_i$ obtained from $\bar X_i'$ by rescaling the metric by a factor
$\lambda_i=1/||\phi_i||_{\bar X'_i}$.

Assume that the stabilizer in $G$ of some arc $[a,b]\subset T$, 
contains a hyperbolic element $h\in G$ whose image   $\phi_i\circ q_i(h)$ is loxodromic in $\bar X'_i$ for infinitely many indices $i$. 
Then  $\stab([a,b])$ is $\Zmax$.
\end{lemma}

\begin{rem}
  The proof of this lemma does not use type preservation. It will be applied without this hypothesis below.
\end{rem}

\begin{proof}

Fix $M\in\bbN$ whose value will be made explicit below.

Consider any $g\in G$ fixing $[a,b]$,
let $a_i,b_i$ be $\eps_i$-approximation points for $a,b$ relative to the action of $g$ and $h^j$, for $j=0,\dots, M$, with $\eps_i\ra 0$.
We view $a_i,b_i$ as points in $\bar X_i'$, and we measure distances in the $\delta_1$-hyperbolic spaces $\bar X_i'$.
Thus, $d_{\bar X'_i}(a_i,b_i)=\lambda_i(d_T(a,b)+o(1))$.
 The action of $G$ on $\bar X'_i$ is via $\phi_i\circ q_i$ but abusing notations, we simply denote it by $gx:=\phi_i\circ q_i(g) x$.
Let $[a'_i,b'_i]\subset [a_i,b_i]$ be the central subinterval of length $\frac23 d_{\bar X_i}(a_i,b_i)$.

Given $j\leq M$,
since the four distances $d_{\bar X'_i}(a_i,ga_i)$, $d_{\bar X'_i}(b_i,gb_i)$, $d_{\bar X'_i}(a_i,h^ja_i)$, $d_{\bar X'_i}(b_i,h^j b_i)$ are small compared to 
$d_{\bar X'_i}(a_i,b_i)$, it follows from Paulin's argument  \cite[p. 341]{Pau_arboreal} (see also \cite[p.284]{BrSw_Paulin})
that the commutator $[g,h^j]$ moves each point of $[a'_i,b'_i]$ by at most $C(\delta_1)$ for some universal constant $C(\delta_1)$.

We now claim that $[a'_i,b'_i]$ intersects $\bar X_i'\setminus \bar\calh^{\bar X'_i}_{20\delta_1}$.
If not, then $[a'_i,b'_i]$ is contained in the trace of a horoball $\bar B\in  \bar\calh^{\bar X'_i}_{20\delta_1}$.
Since for $i$ large enough, $d_{\bar X'_i}(a'_i,ha'_i)$ and  $d_{\bar X'_i}(b'_i,hb'_i)$ are small compared to $d_{\bar X'_i}(a'_i,b'_i)$, 
hyperbolicity of the corresponding quadrilateral shows that  $h [a'_i,b'_i]\subset h \bar B$ contains a point in the $10\delta_1$-neighbourhood of 
 $[a'_i,b'_i]\subset \bar B$. Thus, since the subsets in $\bar\calh^{\bar X'_i}_{20\delta_1}$ are $20\delta_1$-separated (Lemma \ref{lem_barsep}), 
this implies that  $h \bar B= \bar B$ and therefore that $h^k \bar B= \bar B$ for all $k\in \bbZ$.
Since $\bar B$ is bounded, this contradicts the hypothesis that $\phi_i\circ q_i(h)$ is loxodromic, and proves our claim.

By Corollary \ref{cor_uniform_properness}, 
if $i$ is large enough, there a bound $M$ (independent of $i$) on the number of elements that move a point
in $\bar X'_i\setminus \bar\calh_{20\delta_1}^{\bar X'_i}$ 
by at most $C(\delta_1)$.
It follows that there exists $j_1<j_2\in \{0,\dots,M\}$ such that
$[g,h^{j_1}]=[g,h^{j_2}]$. This implies that $g$ commutes with $h^{j_2-j_1}$.
Since this holds for any $g\in \Stab [a,b]$, we get that $h^{M!}$ is central in $\Stab [a,b]$, and that $\Stab [a,b]$ is virtually cyclic with infinite center.

Finally, there remains to check that $\Stab [a,b]$ is $\Zmax$, \ie that any $g\in G$ centralizing $h^k$
for some $k>0$ fixes $[a,b]$. The argument is similar to the one in
\cite{DG2} (proof of Proposition 3.1, page 255), and we leave it to
the reader. 
\end{proof}

\subsubsection{Variation with small peripheral  subgroups}\label{sec_variation}

In this section, we prove that Proposition \ref{prop;R-tree} and its corollary \ref{cor;R-tree} still hold
if we don't assume the monomophisms $\phi_i$ to be type preserving, but require that the groups in $\calp$
are \emph{small}, \ie don't contain a non-abelian free subgroup.  We also need to replace $\Zmax$-rigidity by plain rigidity (Definition \ref{dfn_rigid}).
%%%
Since parabolic groups are small, 
it amounts to requiring that $G$ does not admit
any non-trivial splitting relative to the parabolic groups with small edge groups.

   \begin{prop}\label{prop;R-tree_small}  
Let $(G, \calP), (G', \calP')$ be two relatively hyperbolic groups
whose peripheral groups are finitely generated and small. 
     Let $K_i, K'_i$ be cofinal sequences of 
     Dehn kernels of two relatively hyperbolic groups $(G, \calP), (G', \calP')$. Using notations \ref{notations},
     assume that for each $i$ there is a (not necessarily type preserving) monomorphism 
     $\phi_i: \DFi \hookrightarrow \DFpi$ such that $||\phi_i||_{\bar X'_i}$ is unbounded.

     Then $G$ admits a non-trivial isometric 
     action on a $\mathbb{R}$-tree $T$ such that
     \begin{itemize}
     \item every peripheral subgroup $P\in \calp$ which is not virtually abelian fixes a point in $T$
       \item every arc stabilizer is  small 
     \end{itemize}
   \end{prop}

We will need the following result to analyze the obtained $\bbR$-tree.

 \begin{theo}\label{thm_rips2} 
Consider $(G,\calp)$  a relatively hyperbolic group where each peripheral group is small.
Assume that $G$ acts without global fixed point on an $\bbR$-tree $T$ with small arc stabilizers
so that each peripheral group  is either virtually abelian, or fixes a point in $T$.
Then $G$ has a non-trivial splitting over a small subgroup.

Moreover, if no peripheral subgroup is virtually cyclic, then $G$ is
not rigid (as in Definition \ref{dfn_rigid}): 
$(G,\calp)$ has a non-trivial splitting over a small subgroup 
such that each peripheral group is 
conjugate into some factor of the splitting.   
\end{theo}

We immediately deduce:

   \begin{cor}\label{cor;R-tree_small}
     Under the hypotheses of Proposition \ref{prop;R-tree_small}, if no peripheral group is virtually cyclic then
     $G$ is not rigid.  
\qed
   \end{cor}

\begin{proof}[Proof of Theorem \ref{thm_rips2}]
     We use results in Rips Theory contained in \cite{GL6}.
Let $\calp_0\subset \calp$ the collection of peripheral subgroups that are not virtually abelian.
Being relatively hyperbolic, $G$ is relatively finitely presented with respect to $\calp$ hence with respect to $\calp_0$. 
The action of $G$ on $T$ is hypostable by \cite[Lemma 9.7]{GL6} (see this paper for the definition); the result does not exactly apply to our situation
because groups in $\calp\setminus \calp_0$ might not be elliptic in $T$, but the proof immediately extends to this context
because the class of subgroups contained in a group in $\calp\setminus \calp_0$ 
satisfies the ascending chain condition. 
By Theorem 9.9 of \cite{GL6}, we get that $G$ has a non-trivial splitting over a small group,  and that the groups in $\calp_0$ are 
conjugate into a factor of this splitting.

Assume first that the splitting obtained is a splitting over a finite group.
Being not virtually cyclic, any peripheral group $P\in\calp$ has to be conjugate in a factor of this splitting
 (a small group can split over a finite group only if it is virtually cyclic),
and $G$ is not rigid.
 If the splitting is over an infinite group that is virtually cyclic or parabolic,
the construction of the tree of cylinders in \cite[example 3.4]{GL4} 
 yields another non-trivial splitting of $G$ over a virtually cyclic or parabolic subgroup, and in
 which every peripheral group is conjugate into a factor  \cite[Proposition 5.3, 6.2]{GL4}. Thus $G$ is not rigid. 
\end{proof}

\begin{proof}[Proof of  Proposition \ref{prop;R-tree_small}]
The argument is identical to the proof of  Proposition \ref{prop;R-tree}, except
for the two following points.
Recall that the $\bbR$-tree $T$ was obtained as a limit of rescalings of the action of $G$ on $\bar X'_i$
through the morphisms $\phi_i\circ q_i$.

The first place where we used type preservation is to prove that every peripheral group in $P\in\calp$ fixes a point in $T$,
so we need a different argument here.
If not, then there exists $g\in P$ acting as a hyperbolic isometry on $T$.
We saw that applying \cite[Lemme 9.2.2]{CDP}, this implies that for all $i$ large enough, $g$ acts as a hyperbolic isometry on $\bar X'_i$.

Thus, the group $H_i=\phi_i\circ q_i(P)$ is a small subgroup of $\bar G_i$, that contains a hyperbolic element.
This implies that $H_i$ fixes a point $\omega$ in the boundary at infinity of $\bar X'_i$.
Let $\rho$ be a geodesic ray ending at $\omega$.

Applying Paulin's argument \cite[p. 341]{Pau_arboreal} (see also \cite[p.284]{BrSw_Paulin}), 
we get that there exists a constant $C(\delta_1)$ such that for every $h_1,h_2\in H_i$, the commutator $[h_1,h_2]$ moves by at most $C(\delta_1)$
all points in a subray $\rho_{[h_1,h_2]}\subset \rho$. 
In particular, every finite set $S$ of commutators of elements of $H_i$ moves by at most $C(\delta_1)$ all points in 
a subray $\rho_{S}\subset \rho$.
Note that $\rho_S$ intersects $\bar X'_i\setminus 
\bar\calh_0^{\bar  X'_i}$,  
so by Corollary \ref{cor_uniform_properness},
there is a bound (independent of $S$ and $i$)  on the cardinality of
these commutators. 
In other words, the set $H'_i$ of all commutators of elements in $H_i$
is  bounded independently of
$i$. 
By \cite[Lemma 1A p334]{Pau_arboreal}, $H_i$ contains an abelian subgroup of index bounded by some number $M$ independant of $i$.
This implies that $P$ is virtually abelian.
Indeed, let $P_0<P$ be the intersection of all subgroups of index at most $M$. Then
$\phi_i\circ q_i(P_0)\simeq q_i(P_0)$ is abelian for all $i$, so $P_0$ is abelian since $K_i$ is cofinal.
This concludes the proof that  every $P\in \calp$ which is not virtually abelian fixes a point in $T$.

The second place where type preservation was used is to prove that the stabilizer $H$ of any arc in the limit tree $T$ is small.
If $H$ is not  small,   it contains a free group $\grp{h_1,h_2}$ all
whose non-trivial elements are hyperbolic  (Lemma \ref{lem_tits}). 
By Lemma \ref{lem_F2}, there exists $N$ such that for $i$ large enough, $q_i$ is injective on $\grp{h_1^N,h_2^N}$.
Since peripheral groups are small, this implies that $\phi_i\circ q_i(\grp{h_1^N,h_2^N})$ is not sub-$\bPF'_i$.
If we can find some element $h\in\grp{h_1^N,h_2^N}$ that acts as a loxodromic isometry on $\bar X'_i$ for infinitely many indices $i$,
then Lemma \ref{lem;arcstab_VC} concludes that $H$ is $\Zmax$, and small in particular.

So assume that there exist no such $h$, and consider any finite subset
$S\subset \grp{h_1^N,h_2^N}$ containing $h_1^N$ and $h_2^N$.  
Our assumption ensures that that for all $i$ large enough, no product of two element of $S$ acts as a loxodromic isometry on $\bar X_i'$. 
By \cite[Proposition 3.2]{Koubi_croissance}, there is a point $x_i\in \bar X_i'$ which is moved by at most $100\delta_1$ by every $s\in S$.
If $x_i\in \bar B$ for some $\bar B\in \bar\calh_{100\delta_1}^{\bar X'_i}$, then $\grp{S}=\grp{h_1^N,h_2^N}$ stabilises $\bar B$ since 
$\bar\calh_{100\delta_1}^{\bar X'_i}$-separated. This contradicts that  $\phi_i \circ q_i(\grp{h_1^N,h_2^N})$ is not sub-$\bPF'_i$.
Thus, $x_i\in \bar X'_i\setminus \bar\calh_{100\delta_1}^{\bar X'_i}$. Now by Corollary \ref{cor_uniform_properness}
there is a bound $M$ such that for all $i$ large enough, there are at
most $M$ elements of $\bar G'_i$ that moves a point $x_i\in \bar
X'_i\setminus \bar\calh_{100\delta_1}^{\bar X'_i}$ by at most
$100\delta_1$. Taking $S$ of cardinality $M+1$ yields a
contradiction.  This concludes the proof that arc stabilizers are small.

\end{proof}

  \subsection{Co-Hopf property}

We now compare the co-Hopf property for a relatively hyperbolic group to the existence of actions on $\mathbb{R}$-trees. 
This part can be compared to \cite{Sela_structure,BeSz_endomorphisms}.

  \begin{prop}\label{prop;cohopf}
    Let $(G,\calp)$ be a finitely generated non-elementary relatively hyperbolic group,
 and $\PF$ the class of parabolic or finite subgroups. 
    Assume that $\phi:(G,\PF)\ra (G,\PF)$ is a type-preserving monomorphism (see Definition \ref{dfn_type}).

    If $\phi$ is not onto, then $G$ has a non-trivial action on an $\bbR$-tree $T$ 
such that every peripheral subgroup $P\in \calp$ fixes a point in $T$
and every arc stabilizer is finite, parabolic or $\Zmax$. 
  \end{prop}

  \begin{rem}
    Note that the assumption that $G$ is non-elementary is essential.
  \end{rem}

   Applying Theorem \ref{thm_rips}, we immediately get a splitting as follows.

   \begin{cor}\label{cor;cohopf}
     Under the hypotheses of Proposition \ref{prop;cohopf}, 
     $G$ is not $\Zmax$-rigid.
     \qed
   \end{cor}

If we assume that  the groups in $\calP$ are small
but not virtually cyclic, we can drop the assumption that $\phi$ is type preserving.
As in subsection \ref{sec_variation}, we only get an action on an $\bbR$-tree with small arc stabilizers as in Theorem \ref{thm_rips2},
hence a contradiction to rigidity.
Note that in this context, for any $P\in \calp$, since $\phi(P)$ is small and not virtually cyclic, it is finite or parabolic by Lemma \ref{lem_tits}.
We will indicate the   places in the proof where the argument changes.

   \begin{cor}\label{cor;cohopf2}
    Let $(G,\calp)$ be a finitely generated non-elementary relatively hyperbolic group.
    Assume that the peripheral groups are small but not virtually cyclic.

    Assume that $\phi:G\ra G$ is a monomorphism (not necessarily type-preserving) that is not onto.
Then $G$ is not rigid.
   \end{cor}

  \begin{proof}[Proof of Proposition \ref{prop;cohopf}]
Our proof follows \cite{BeSz_endomorphisms}.
    Let $X$ be a hyperbolic space with a proper action of $G$ as in Definition \ref{def;RHG}.
Let  $S$ be a generating set of $G$.
We consider iterates $\phi^i$ of $\phi$, and define $||\phi^i||_X=\displ_X(\phi^i(S))$ 
(see Section \ref{sec_bounded} for the definition of $\displ$).
    If $||\phi^i||_X$ is unbounded, we can argue as in Section \ref{sec_unbounded}, and produce
an action on an $\bbR$-tree  satisfying the hypotheses of Theorem \ref{thm_rips} (or Theorem \ref{thm_rips2}
in the context of Corollary \ref{cor;cohopf2}).

Assume now that $||\phi^i||_X\leq M$ for some $M>0$. Let $x_i\in X$ be a point moved at distance at most $2M$
by $\phi^i(S)$, and let $\calh_{2M}$ be a $2M$-separated system of horoballs.
If $x_i$ lies in some horoball $B\in \calh_{2M}$, then for each $s\in S$, $\phi^i(s)$ preserves this horoball.
Since $S$ generates $G$, $\phi^i(G)$ preserves this horoball, and consists of parabolic elements.
This contradicts the fact that $G$ contains loxodromic elements (because it is non-elementary)
and that $\phi$ is type-preserving. In the context of 
 Corollary \ref{cor;cohopf2} where one assumes that the peripheral groups are small, this contradicts the fact that $G$ itself contains a free subgroup because it is non-elementary.

It follows that $x_i\in X\setminus \calh_{2M}$.
Since $G$ acts cocompactly on  $X\setminus \calh_{2M}$, 
there is a compact subset  $K\subset X$ and $h_i\in G$
such that $\ad_{h_i}\circ \phi^i(S)$ moves some point $x_i\in K$ by at most $M$.
Since the action of $G$ on $X$ is proper, there exists $i< j$ such that for all $s\in S$,
$\ad_{h_i}\circ \phi^i(s)=\ad_{h_j}\circ \phi^j(s)$.
We thus get that $\phi^j=\ad_{h_j\m h_i}\circ \phi^i$.

First for all $i$, since $\phi^i(G)$ is not an elementary subgroup of $(G, \calP)$, it contains independant hyperbolic elements 
$g_i,g'_i$.  
This implies that the centralizer $A_i$ of $\phi^i(G)$ is finite for all $i$. 

We claim that $\#A_i$ is bounded.
Indeed, let $F\subset X$ be the set of points of $X$ moved by by at most $100\delta$ by all elements of $A_i$.
By \cite[Proposition 3.2]{Koubi_croissance}, $F$ is non-empty.
Moreover, for all $a,b\in F$, $[a,b]\subset F'$ where $F'$
is the set of points of $X$ moved by by at most $200\delta$ by all elements of $A_i$.
Since $\phi^i(G)$ is non-elementary, it contains a hyperbolic element $g$.
Fix $a\in F$, and consider a segment $[a,ga]$. It has to intersect $X\setminus \calh_{0}$
since otherwise, $[a,ga]$ would be contained in a horoball $B\in \calh$, 
and $B$ would be $g$-invariant since the horoballs in $\calh_0$ are disjoint.
It follows that $F'$ contains a point in $X\setminus \calh_0$.
Since $X\setminus \calh$ is cocompact, properness of the action gives a bound on $\#A_i$.
This proves the claim.

 Since $A_i\subset A_{i+1}$, the groups $A_i$ eventually stabilize to a finite group $A$.
Then we can argue as in \cite[Theorem 3.1]{RiSe_structure} (see also
\cite[Corollary 2.2]{BeSz_endomorphisms}). 
Assume that $i$ is large enough so that $A_i=A$
and  consider $i<j$ and $g\in G$ such that $\phi^{j}=\ad_g\circ
\phi^i$. For all $k>0$,
 we have 
 $\phi^k\phi^j=\ad_{\phi^k(g)}\circ\phi^{i+k}$
 and $\phi^j\phi^k=\ad_g\circ\phi^{i+k}$, so $g\m\phi^k(g)\in A$.
 Since $A$ is finite, there exist $k<l $ such that $g\m\phi^k(g)=g\m\phi^l(g)$, so
 $\phi^k(g)=\phi^l(g)$ is fixed by $\phi^{l-k}$. This implies that $\phi^k(g)\in \phi^{k'}(G)$ for all $k'>0$.
In particular, writing $\phi^k(g)=\phi^{i+k}(h)$, we have
 $\phi^{j+k}=\ad_{\phi^k(g)}\circ\phi^{i+k}=\ad_{\phi^{i+k}(h)}\circ\phi^{i+k}=\phi^{i+k}\circ \ad_h$,
so $\phi^{j+k}(G)=\phi^{i+k}(G)$, a contradiction.
   \end{proof}

\subsection{ Proof of the main rigidity theorem, and variations}
\label{sec_DF_carac}

We now prove Theorem \ref{thm_DF_carac}, then we will formulate a
variation of it without type preservation, and we will give a number of
corollaries.

\begin{proof}[Proof of Theorem \ref{thm_DF_carac}]
Let $\phi_i:(G/K_i, \bPF_i)\ra (G'/K'_i,\bPF'_i)$ be a sequence of type preserving isomorphisms
perserving the marked peripheral structures. 
We use Notations \ref{notations}. 

If $||\phi_i||_{\bar X'_{i}}$ is unbounded,  Corollary \ref{cor;R-tree} provides a splitting of $G'$ 
 contradicting our assumption.
Thus $||\phi_i||_{\bar X'_{i}}$ is bounded.
By Proposition \ref{prop;mono}, there exists a type-preserving monomorphism 
$\phi:(G,\PF)\ra (G',\PF')$.

By symmetry of the argument, there also exists a type-preserving  monomorphism $\psi:(G,\PF)\ra (G',\PF')$.
Then by Corollary \ref{cor;cohopf}, 
$\phi\circ\psi$ and $\psi\circ \phi$ are onto and so 
are $\phi$ and $\psi$.
Thus $\phi$ is the desired type-preserving isomorphism.
\end{proof}

The variations of the previous section allow  to state and prove  a
version of Theorem \ref{thm_DF_carac} where one drops the assumption
that isomorphisms between Dehn fillings are type preserving. This
provides  a generalisation of the classical  fact that two finite volume hyperbolic 3-manifolds with same  family of classical Dehn fillings are isometric.

\begin{theo}\label{theo;variation_small}
Let $(G,\calP), (G',\calP')$ be two relatively hyperbolic groups,
where each group in $\calp\cup\calp'$ is infinite, small, but not
virtually cyclic. Assume that 
 $(G,\calP)$  and $ (G',\calP')$ are rigid   
 and non-elementary.

Assume that there exist cofinal sequences of Dehn kernels $K_i\normal G$, $K'_i\normal G'$, 
such that for each $i$, there  is an  isomorphism $\varphi_i:  G'/K'_i \to G/K_i$. 

Then there is an isomorphism $G\ra G'$ sending $\calp$ to $\calp'$.
\end{theo}

\begin{proof}[Proof sketch]
  The proof is essentially the same as the proof of Theorem
  \ref{thm_DF_carac}.  Using Corollary \ref{cor;R-tree_small} instead
  of Corollary \ref{cor;R-tree} we can assume that the scaling factors
  are bounded.  Using the variation on Proposition \ref{prop;mono}
  given in Remark \ref{rem;wo_type_preservation_inDF}, we get the
  existence of a (non-type preserving) monomorphism $\phi:G\ra G'$,
  and another one $\psi:G'\ra G$ by symmetry of the argument.  By
  Corollary \ref{cor;cohopf2}, $\phi\circ \psi$ and $\psi\circ \phi$
  are onto, so $\phi$ and $\psi$ are isomorphisms.  Finally, since
  every peripheral subgroup is small, infinite and not virtually
  cyclic, Lemma \ref{lem_tits} shows that parabolic groups are
  characterized as maximal non-virtually cyclic small subgroups.  This
  implies that $\phi(\calp)=\calp'$.
\end{proof}

There are diverse corollaries of this theorem. 
The first one is  on how cofinal sets of Dehn fillings determine the group.

\begin{cor}\label{cor_cofinal}
  Let $(G,\calP), (G',\calP')$ be two relatively hyperbolic groups, where each group in $\calp\cup\calp'$ is infinite, small, but not virtually cyclic.
Assume both are non-elementary and rigid.

Let $\calc,\calc'$ be the set of isomorphism classes  (without peripheral structure) 
of two infinite cofinal sets of Dehn fillings of $(G,\calp)$ and $(G',\calp')$.

If $\calc=\calc'$, then there is an isomorphism $G\ra G'$ sending $\calp$ to $\calp'$.
\end{cor}

\begin{proof} %%%
Let $\calq,\calq'$ be two infinite cofinal sets of Dehn kernels
of $(G,\calp)$ and $(G',\calp')$ such that for any $K\in \calq$ there is  some $K'\in\calq'$  such that $G/K\simeq G'/K'$, and conversely.

Let $K_i\in \calq$ be an   infinite sequence of distinct Dehn
kernels. 
Since $\calq$ is cofinal, this is a cofinal sequence of Dehn kernels.
For each $i$, consider $K'_i\in \calq'$ such that $G/K_i\simeq G'/K'_i$.

Assume first that as $i$ varies, $K'_i$ takes  infinitely many distinct values.
Then up to taking a subsequence, we can assume that all $K'_i$ are distinct,
and since $\calq'$ is cofinal, $K'_i$ is also a cofinal sequence of Dehn kernels.
Theorem \ref{theo;variation_small} then applies and the corollary is proved.

Assume on the contrary that $K'_i$ takes only finitely many values
so that, up to taking a subsequence, we can assume
that all the groups $G/K_i$ are isomorphic.
 We don't know whether $K'_i$ is a proper Dehn kernel,
but there exists $i_0$ large enough so that $K_{i_0}$ is a proper Dehn kernel, and in particular 
 the group $G''=G/K_{i_0}$ is hyperbolic relative to the image $\calp''$ of $\calp$.
We are going to apply our rigidity results to the sequence $q_i:G\ra G/K_i$
and the constant trivial sequence $q_i''=\id :G''\ra G''=G''/K''_i$
where $K''_i=\{1\}$ is a (trivial!) cofinal (constant) sequence of
Dehn kernels of $(G'',\calp'')$. 
%%%
Since $(G,\calp)$ is rigid, Corollary \ref{cor;R-tree_small}
implies that $||\phi_i||_{\bar X''_i}$ is bounded.
By Proposition \ref{prop;mono} and its variation stated in Remark \ref{rem;wo_type_preservation_inDF},
one gets a (maybe not type preserving) 
monomorphism $\phi:G\ra G''$, that makes the following diagram commute for $i$ large enough
$$\xymatrix@R=8mm@C=15mm@M=2mm{G\ar@{^{(}->}[r]^{\phi} \ar@{->>}[d]^(0.4){q_i} & G'' \ar@{->>}[d]^(0.4){q''_i}
\\ \DFi \ar@{^{(}->}[r]^{\ad_{h_i}\circ\phi_i} & \DFpi }$$
%$q''_i\circ \phi=\ad_{h_i}\circ \phi_i\circ q_i$ 
for some $h_i\in G''_i$.
Since $q''_i\circ\phi$ is injective, the map $q_i$ must also be injective.  
  This implies that $K_i=\{1\}$,  contradicting that all $K_i$ are distinct.
\end{proof}

As an application of the previous corollary we may cite the
case of characteristic Dehn fillings.

\begin{dfn}\label{def;CDK}
  If $P$ is finitely generated and residually finite, and $i\geq 1$, the $i$-th \emph{characteristic core $C_i(P)$} is the intersection
of all subgroups  of index at most $i$ of $P$.

For each $i\geq 1$, we define the $i$-th \emph{characteristic Dehn kernels} of $(G,\calp)$
by
$$K_i=\ngrp{C_i(P_1),\dots, C_i(P_k)}\normal G.$$
\end{dfn}

This notion  provides a natural  application of the
previous corollary: if the peripheral subgroups are residually
finite, the sequence of characteristic cores $C_i(P)$ is  cofinal in $P$, hence, by Lemma
\ref{lem;cof},  so is the sequence of characteristic Dehn kernels $K_i$ in $G$. 
Applying Corollary \ref{cor_cofinal}, we get the following result which applies in particular when
peripheral groups are virtually polycyclic.

\begin{cor}
\label{cor_carac}
Let $(G,\calp),(G',\calp')$ be two non-elementary, rigid relatively hyperbolic groups whose peripheral subgroups are small and residually finite
but not virtually cyclic or finite.

If   $(G,\calp)$ and $(G',\calp')$ have the same isomorphism classes of characteristic Dehn
fillings (without peripheral structure), then $(G,\calp)\simeq (G',\calp')$.\qed
\end{cor}

Another corollary says that the collection of \emph{proper} finite Dehn fillings determine the group.

\begin{cor}\label{cor_proper}
Consider $(G,\calp)$, $(G',\calp')$ two relatively hyperbolic groups, whose peripheral subgroups are 
finitely generated, residually finite,  small, but are not finite nor virtually cyclic.
Assume that $G$ and $G'$ are non-elementary, and rigid. 

Assume that 
 any deep enough finite Dehn filling of any of the two relatively
 hyperbolic groups 
 is isomorphic to some \emph{proper} finite Dehn filling of the other.

Then there is an isomorphism $G\ra G'$ sending $\calp$ to $\calp'$.
\end{cor}

\begin{proof}
Write $\calp=\{[P_1],\dots,[P_k]\}$, and $\calp'=\{[P'_1],\dots,[P'_{k'}]\}$.

 We first prove that $k=k'$.
Given a group $H$ and $d\geq 1$, define  $n_d(H)\in\bbN\cup\infty$ as the number of conjugacy classes 
of maximal finite subgroups of $H$ of cardinality $\geq d$ (a maximal finite subgroup
is a finite subgroup not contained in any larger finite subgroup).
Let $d_0(G)$ be the maximal cardinality of a non-parabolic finite subgroup of $G$.
We first claim that if $\bar G$ is any proper Dehn filling of $G$, 
then for all $d>d_0(G)$, $n_d(\bar G)\leq k$. Indeed, by Lemma \ref{lem;cof}, 
any finite subgroup of cardinality $> d_0(G)$ lies in the image of a
parabolic subgroup which proves our claim. 
Moreover, we claim that for any $d> d_0(G)$, any deep enough Dehn filling of $G$ satisfies that
$n_d(\bar G)=k$.
Indeed, since peripheral groups of $G$ are infinite and residually finite, the image of $P_i$ in  any deep enough Dehn filling $\bar G$ of $G$ has cardinality $>d_0(G)$,
and Lemma \ref{lem_preserv_non_conj} ensures that they are not conjugate.
This implies that $n_d(\bar G)\geq k$ and proves our second claim. 
Now take $d> \max(d_0(G),d_0(G'))$. Our second claim says that there is a Dehn filling $\bar G$
with $k=n_d(\bar G)$. Let $\bar G'$ be some proper Dehn filling of $G'$ isomorphic to $\bar G$.
Then $n_d(\bar G)=n_d(\bar G')\leq k'$ by our first claim.
We thus deduce $k\leq k'$, and $k=k'$ by symmetry of the argument.

Given a finitely generated group $P$, denote by 
 $C_i(P)$ and $K_i$ the $i$-th characteristic core of $P$, and the
 $i$-th characteristic Dehn kernel introduced above
 (Def. \ref{def;CDK}). 
Our assumption says that for $i$ large enough, there exists a proper Dehn filling 
$\bar G'_i=G'/\ngrp{N_1'^i,\dots,N_{k'}'^i}$.
 isomorphic to $ G/K_i$.

We claim that $N_1'^i\supset C_i(P'_1)$ for $i$ large enough.
Indeed,  the finite group $P_1/C_i(P_1)$ is $i$-separated in the sense that the intersection of all its 
subgroups of index $\leq i$ is trivial.
Since $P_j/C_i(P_j)$ is a maximal finite subgroup of cardinality $\geq d_0(G')$,
there is a permutation $\sigma_i$ of $\{1,\dots,k\}$
such that $P_j/C_i(P_j)$  is isomorphic to $P'_{\sigma_i(j)}/N'_{\sigma_i(j)}$ for all $j$.
Therefore, every group $P'_{j'}/N'_{j'}$ is $i$-separated,
which implies that $C_i(P'_{j'})\subset N'_{j'}$.
In particular, $\#P_j/C_i(P_j) =\#P'_{\sigma_i(j)}/N'_{j'} \leq \#P'_{\sigma_i(j)}/C_i(P'_{\sigma_i(j)})$.
By symmetry of the argument, we get that there is equality of cardinalities,
so that $N'_{j'}= C_i(P'_{j'})$, and $\bar G'_i=G'/K'_i$ is the $i$-th characteristic Dehn filling of $G'$.

Since peripheral groups are residually finite, $G/K_i$ and $G'/K'_i$ are cofinal sequences of Dehn fillings,
 Theorem \ref{theo;variation_small} says that there is an isomorphism  $G\ra G'$
 sending $\calp$ to $\calp'$.
\end{proof}

Yet another corollary is on how finite Dehn fillings with peripheral
structures determine the group with peripheral structure.

\begin{cor}\label{cor_periph}
Consider $(G,\calp)$, $(G',\calp')$ two relatively hyperbolic groups, whose peripheral groups are 
finitely generated, residually finite, but are not finite.
Assume that $G$ and $G'$ are non-elementary, and have no elementary splittings relative to $\calp$ (resp. $\calp'$).

Assume that $G$ and $G'$ have the same collection of finite Dehn fillings, viewed as groups with a peripheral structure.
More precisely, assume that for any finite Dehn kernel $K$ of $(G,\calp)$
there exists a finite Dehn kernel $K'$ of $(G',\calp')$ such that $(G/K,\bar \calp)\simeq (G'/K',\bar \calp')$, 
and conversely.  %%%

Then there is an isomorphism $G\ra G'$ sending $\calp$ to $\calp'$.
\end{cor}

\begin{proof}
Consider $K_i$ the $i$-th characteristic Dehn kernel of $G$ (see Definition \ref{def;CDK}).
Our assumption now says that for all $i$, there exists a (maybe non-proper) finite Dehn kernel 
$K'_i=\ngrp{N_1'^i,\dots,N_{k'}'^i}$ and an isomorphism $\phi_i: G/K_i\ra G'/K'_i$
sending the image $\bar\calp_i$ of $\calp$ in $G/K_i$ to the image $\bar\calp'_i$ of $\calp'$ in $G'/K'_i$.

We know (Lemma \ref{lem_preserv_non_conj})  that when $i$
is large enough, $\bar\calp_i$ consists of exactly $k$ conjugacy
classes of groups.

On the other hand, $\bar\calp'_i$ consists of at most $k'$ conjugacy classes.
It follows that $k\leq k'$, hence $k=k'$ by symmetry of the argument.

Since we don't know that the Dehn filling is proper, the group $M'^i_j=P'_j\cap K'_i$ might be
larger than $N'^i_j$. 
The group $P'_j/M'^i_j$ is one of the groups in $\bar \calp'_i$, and by hypothesis,
$\phi_i\m$ induces an isomorphism with some conjugate of some $P_{\sigma_i(j)}/C_i(P_{\sigma_i(j)})$
for some bijection $\sigma_i$ of $\{1,\dots,k\}$.
Up to going to a subsequence, we can assume that the permutation $\sigma_i$ does not depend on $i$,
and we denote it by $\sigma$.
Since $P'_j/M'^i_j$ is isomorphic to $P_{\sigma(j)}/C_i(P_{\sigma(j)})$,
it is $i$-separated (in the sense used in the previous proof). It follows that $C_i(P'_j)\subset M'^i_j$,
and that $\# P_{\sigma(j)}/C_i(P_{\sigma(j)})=\#P'_j/M'^i_j\leq \# P'_j/C_i(P'_j)$.

Starting from the characteristic Dehn fillings defined by  $P'_j/C_i(P'_j)$ for the subsequence defined above,
we can reverse the argument and find a further subsequence and another permutation $\sigma'$ such that 
for all $j\in \{1,\dots k$, $\# P'_{\sigma'(j)}/C_i(P'_{\sigma'(j)})\leq \# P_j/C_i(P_j)$.
Combining these, we get that all these inequalities are in fact equalities 
which implies that $C_i(P'_j)= M'^i_j$.

Thus $K'_i$ is a characteristic Dehn kernel, $K'_i$ is a cofinal sequence, and
Theorem \ref{thm_DF_carac} applies.
\end{proof}

\section{Solution to the isomorphism problem}

In this section we use the fact proved in the previous section
that rigid relatively hyperbolic groups are determined by their Dehn fillings
to give an algorithmic solution to the isomorphism problem when parabolic groups are residually finite.

We first state a version in which the peripheral subgroups are given in the input.
In this algorithm, each relatively hyperbolic group $(G,\calp)$ is given as follows.
The group $G$ is given by some finite presentation $G=\grp{S|R}$.
Recall that by definition $\calp$ is stable under conjugation, and choose $P_1,\dots, P_k$ 
some representatives of the conjugacy classes.
Then $\calp$ is given by a choice of generating set of each $P_i$, each generator being given as a word on $S^{\pm 1}$.
Recall that $\Zmax$-rigidity is a weakening of rigidity introduced in Definition \ref{dfn_Zmax_rigid}.

  \begin{theo} \label{theo;isom_algo}
    There is an algorithm that solves the following problem.

    The input is  a pair of relatively hyperbolic groups $(G,\calp),(G',\calp')$ given by
    finite presentations $G=\grp{S|R},G'=\grp{S'|R'}$
    together with finite generating sets of conjugacy representatives of $\calp$ and $\calp'$ as above.
%%%
We assume
%%%
%%%
%%%
 that the following assumptions hold:
      \begin{itemize}
       \item $(G,\calp)$ and $(G',\calp')$ are $\Zmax$-rigid and non-elementary
      \item peripheral subgroups are infinite, and residually finite
%%%
      \end{itemize}

    The output is the answer to the question  whether $(G,\calP)\simeq(G,\calp')$, ie whether there exists an isomorphism sending $\calp$ to $\calp'$.
    If the answer is positive, the algorithm also gives an explicit isomorphism.
%%%
  \end{theo}

  \begin{rem}
    Without assuming that peripheral groups are infinite, our algorithm works and says whether there is a type preserving
    isomorphism $(G,\PF)\simeq (G',\PF')$ (see Section \ref{sec_type}).
  \end{rem}

In section \ref{sec_notgiven}, we will give a variant where we don't give the peripheral subgroups in the input.

The proof of the theorem will be given in the next subsections.

Here is an overview of our algorithm.
Consider two relatively hyperbolic groups $(G,\calp),(G',\calp')$ whose peripheral subgroups are residually finite.
Then consider a cofinal sequence of Dehn fillings $G_i/K_i,G'_i/K'_i$ obtained by killing suitable characteristic subgroups of finite index
of the peripheral subgroups.
This choice of Dehn fillings guarantees that if $(G,\calp)\simeq(G',\calp')$
%%%
then $(G/K_i,\bar \calp_i)\simeq(G'/K'_i,\bar \calp'_i)$ for all $i$,
where  $\bar \calp_i$ and $\bar \calp'_i$ denote the image  of $\calp$ and $\calp'$ in $G/K_i$ and $G'/K'_i$ respectively.

The Dehn filling theorem ensures that the quotient groups $G/K_i$ are Gromov-hyperbolic (not relatively) for $i$ large enough.
Using the solution to the isomorphism problem for hyperbolic groups \cite{DG2},
one can check whether $(G/K_i,\bar\calp_i)$ is isomorphic to $(G/K_i,\bar\calp_i)$.
 Since there exists an algorithm that stops if and only the group given as input is hyperbolic \cite{Papasoglu_algorithm},
we can construct a first algorithm that stops if and only if for some $i$, $G/K_i$ and $G'/K'_i$ are hyperbolic 
and $(G/K_i,\bar \calp_i)\not\simeq(G'/K'_i,\bar \calp'_i)$.
Our rigidity Theorem \ref{thm_DF_carac} implies that this happens if and only if $(G,\calp)\not\simeq (G',\calp')$.
To conclude, it is enough to produce a second algorithm that stops if and only if $(G,\calp)\simeq (G',\calp')$.
This can easily be done  by enumerating all presentations of $G$ using Tietze transformations.

\subsection{More on characteristic Dehn kernels}

Consider $(G,\{[P_1],\dots,[P_k]\})$ a relatively hyperbolic group
with residually finite peripheral subgroups.
In  Definition \ref{def;CDK} we introduced a useful canonical cofinal
family of Dehn kernels, the characteristic Dehn kernels $K_i$:
the $i$-th characteristic core $C_i(P)$ of a group $P$ is %%%
the intersection of all subgroups of  index at most $i$ in $P$, and $K_i=\ngrp{C_i(P_1),\dots,C_i(P_k)}$.
%%%
%%%

 We will need to compute $C_i(P)$ thanks to the following Lemma.

\begin{lem}\label{lem_compute_CiP} 
Given a finite presentation $\grp{S|R}$ of  a group $P$ and $i\geq 1$, one can 
compute a generating set of $C_i(P)$ (given as a finite set of words on $S^{\pm 1}$).
\end{lem}

\begin{proof}
Using the given presentation, one can list all the morphisms $\phi_1,\dots,\phi_k$ from $P$ to the symmetric group $Sym_i$. 
One easily checks that $C_i(P)$ is the intersection of the kernels of all such morphisms,
so $C_i(P)$ is the kernel of the product morphism $\Phi=(\phi_1,\dots\phi_k)$. 
Applying the Reidemeister-Schreier method \cite[Proposition II.4.1]{LyndonSchupp}
then yields a finite generating set (and even a presentation) for $C_i(P)$.
\end{proof}

%%%
The following simple observation applies for all $i$ (even for small values of $i$ where the Dehn kernel $K_i$ is not proper and $G/K_i$ need not be hyperbolic).
We denote by $\bar \calp_i$ the image of $\calp$ in $ G/K_i$, and use similar notations for $G'$.

\begin{lem}\label{lem_characteristic}
  Let $(G,\calp)$, $(G',\calp')$ be relatively hyperbolic groups. 
Let $K_i,K'_i$ be the corresponding  $i$-th characteristic Dehn kernels.

If $(G,\calp)\simeq (G',\calp')$ then for all $i\geq 1$, $(G/K_i,\bar\calp_i)\simeq(G'/K'_i,\bar\calp'_i)$.
\end{lem}

\begin{proof}
  If $\phi:G\ra G'$ is an isomorphism that maps $\calp$ to $\calp'$, then for each $P\in \calp$, 
$\phi(C_i(P))=C_i(\phi(P))=C_i(P')$ where $P'=\phi(P)\in\calp'$. %%%
It follows that $\phi(K_i)=K'_i$, so $\phi$ induces an isomorphism $\bar\phi_i:G/K_i\ra G'/K'_i$
that sends $\bar\calp_i$ to $\bar \calp'_i$.
\end{proof}

Assuming that peripheral groups are residually finite, Lemma \ref{lem;cof} implies that $(K_i)_{i\geq 1}$ is a cofinal sequence of Dehn kernels.
In particular, $K_i$ is a proper Dehn kernel for $i$ large enough, and in particular $G/K_i$ is Gromov hyperbolic (Corollary \ref{cor_hyp}).
We thus get:

\begin{lemma}\label{lem;residual_hyperbolicity}
Any relatively hyperbolic group  $(G,\calP)$ with residually finite peripheral subgroups is fully residually hyperbolic:
for any finite subset $A\subset G\setminus\{1\}$ there exists a hyperbolic group $H$ and a morphism $\phi:G\ra H$
such that $1\notin\phi(A)$.\qed
\end{lemma}

It is well known that the word problem in a relatively hyperbolic group is solvable if and only if
it is solvable in its peripheral subgroups \cite{Farb_relatively}.  
We need a uniform solution 
of the word problem among relatively hyperbolic groups with residually finite peripheral groups, even if the peripheral groups are not explicitly given.

\begin{cor}\label{cor_wordpb}
  There is an algorithm that takes as input $(\grp{S|R},w)$, 
where $\grp{S|R}$ is a finite presentation of a group $G$, such that $G$ is 
relatively hyperbolic with respect to some family of residually finite subgroups, 
and where $w$ in a word on the generating set $S$, and which says whether $w$ represents the trivial element in $G$.
\end{cor}

\begin{proof}
On the one hand, enumerate all words that are products of conjugates of relators in $\grp{S|R}$, and check whether $w$ appears.
 If it does, then $w$ represents the trivial element.

On the other hand, enumerate all group  presentations of all hyperbolic groups. 
This can be done by enumerating all presentations, and using Papasoglu's algorithm that stops only if the group defined by this presentation is hyperbolic
 \cite{Papasoglu_algorithm}.
In this case, Papasoglu's algorithm also provides an explicit linear upper bound on the isoperimetry function, hence a solution to the word problem.
For each such hyperbolic group $H$, enumerate (in parallel) all the morphisms from $G$ to $H$, 
(using the solution to the word problem in $H$), and check whether $w$ has non trivial image. 
 In this case, then $w$ in non-trivial in $G$.

By residual hyperbolicity proved in  Lemma \ref{lem;residual_hyperbolicity}, one of the two procedures must terminate, 
and allows one to decide whether $w$ is trivial in $G$ or not.
\end{proof}

\subsection{Solution to the isomorphism problem when peripheral subgroups are given}

We are now ready to prove Theorem \ref{theo;isom_algo} solving the isomorphism problem for rigid
relatively hyperbolic groups with residually finite peripheral subgroups.

The first step in the proof of the theorem is the following result
that allows one to compute a presentation  
 of the peripheral subgroups
from their generating set.

\begin{prop}\label{prop;presentation_des_periph}
 There exists an algorithm that solves the following problem.

The input is a relatively hyperbolic group $(G,\{[P_1],\dots,[P_k]\})$ with the assumption that each $P_i$ is residually finite.
It is given to the algorithm as in Theorem \ref{theo;isom_algo} by a finite presentation of $G$, 
and a finite generating set $S_i$ for each $P_i$. 

The output, 
is a finite presentation for each $P_i$ on the same generating set $S_i$.
\end{prop} 

\begin{proof}
Under these assumptions, we can solve the word problem by Corollary \ref{cor_wordpb}.
Then we may apply  \cite[Theorem 2]{DG_presenting}, to get precisely the desired result.
\end{proof}

  \begin{proof}[Proof of Theorem \ref{theo;isom_algo}]
Denote by $S_1,\dots S_k$ (resp.\ $S'_1,\dots,S'_{k}$) the generating sets of the peripheral subgroups we are given,
and let $P_i=\grp{S_i}$ (resp.\ $P'_i=\grp{S_i'}$) the corresponding subgroup of $G$ (resp.\ $G'$).
Thanks to Proposition \ref{prop;presentation_des_periph}, we can compute a finite presentation $P_i=\grp{S_i|R_i}$ (resp $P'_j=\grp{S'_j|R'_j}$) of each peripheral group. 

 We run two procedures, $A$ and $B$ in parallel.

   Using Tietze transformations, Procedure $A$ enumerates all presentations of $G$ 
  together with all possible finite generating sets for conjugates of
  the groups $P_i$.
  Then Procedure $A$ stops if this presentation of $G$ obtained coincides with the presentation of $G'$ we were given,
and the generating sets for $P_1,\dots, P_k$ coincide with
$S'_1,\dots, S'_{k}$ up to renumbering. 
In this case, Procedure A correctly says that $(G,\calp)\simeq(G',\calp')$. Moreover, 
if $(G,\calp)\simeq(G',\calp')$, then Procedure A will detect it (\cite[II.~Proposition~2.1]{LyndonSchupp}).
%%%

 Procedure $B$, runs in parallel. 
It starts iteratively for each $i\geq 0$ a subprocedure $B_i$ that runs in parallel for each $i$.
Using Lemma \ref{lem_compute_CiP} and the presentation of $P_i$, 
Procedure $B_i$ computes a generating set the characteristic cores $C_i(P_j)$, $C_i(P'_j)$ of each peripheral subgroup.
We thus get a finite presentation of the characteristic Dehn fillings $G/K_i$ and $G'/K'_i$,
where $K_i=\ngrp{C_i(P_1),\dots, C_i(P_k)}$ and
$K'_i=\ngrp{C_i(P'_1),\dots, C_i(P'_k)}$ are the $i$-th characteristic
Dehn Kernels.  
Then procedure $B_i$ runs Papasoglu's algorithm that stops if and only if the given presentation defines a hyperbolic group
(thus procedure $B_i$ never stops if $G/K_i$ or $G'/K'_i$ is not hyperbolic) \cite{Papasoglu_algorithm}.
If Papasoglu's algorithm stops for both groups, procedure $B_i$
then runs a solution to the isomorphism problem for the class of  
 hyperbolic groups (with torsion),
specifically the algorithm of 
 \cite[Theorem 1]{DG2}  in order to determine whether there is an
 isomorphism $\phi_i: G/K_i \simeq G'/K_i $.  
If some procedure $B_i$ terminates and finds out that
$G/K_i \not\simeq G'/K'_i$, 
then procedure $B$ stops, and says that $ (G,\calp) \not\simeq (G',\calp') $. This is correct by Lemma \ref{lem_characteristic}.
 On the contrary, if for every $i$, Procedure $B_i$ does not stop, or finds out that $G/K_i \simeq G'/K'_i$, then Procedure $B$ does not stop.

To prove the theorem, 
we need to check that  one of the two procedures $A$ and $B$ eventually
 stops.
 Assume that procedure $B$ never stops, \ie that no procedure $B_i$ stops saying that  $G/K_i \not\simeq G'/K_i$.

%%%

%%%
%%%
Since peripheral subgroups are residually finite, the two sequences $K_i$ and $K'_i$ are cofinal
(Lemma \ref{lem;cof}).
Moreover, there exists $i_0$ such that for all $i\geq i_0$, $G/K_i$ and $G'/K'_i$ are hyperbolic.
In particular, procedure $B_i$ stops for all $i\geq i_0$. Since $B$ does not stop, 
for all $i\geq i_0$, there exists an isomorphism $\phi_i: G/K_i \simeq G'/K'_i$.
  Lemma \ref{lem;elli} says that for all $i$  large enough, 
$\phi_i$ is necessarily a type-preserving isomorphism between  $(G/K_i,\bPF_i)$ and $(G'/K'_i,\bPF'_i)$
since for deep enough Dehn fillings, a group is sub-$\bPF_i$ (resp.\ sub-$\bPF_i'$) if and only if it is a finite subgroup of $G/K_i$ (resp $G'/K'_i$).
%%%
%%%
%%%
We can therefore apply our rigidity Theorem \ref{thm_DF_carac} and get that
%%%
there exists type preserving isomorphism $(G,\PF) \to (G',\PF')$.
Since all peripheral groups are infinite, this means by Remark \ref{rk_type} that the isomorphism sends $\calp$ to $\calp'$,
\ie $(G,\calp)\simeq (G',\calp')$.
In this case, Procedure $A$ has to stop.
 \end{proof}

\subsection{When peripheral subgroups are not given}
\label{sec_notgiven}
%%%
 %%%
When peripheral  subgroups are not given to the algorithm, one can try to find them.
For this to work, we need the peripheral subgroups to lie in a given suitable recursively enumerable class $\calc$.

\begin{dfn}
  We say that a set $\calc$ of isomorphism classes of finitely presented groups is recursively enumerable if
there exists a Turing machine that enumerates all finite presentations of all groups in $\calc$.
\end{dfn}

Equivalently, $\calc$ is enumerable if there is a Turing machine enumerating some finite presentations,
each of which represents a group in $\calc$, and such that every group in $\calc$ has at least one presentation
that is enumerated. We say that such a Turing machine \emph{enumerates} $\calc$.

An important example for $\calc$ is the class $VPC_{\geq2}$ of virtually polycyclic groups that are infinite, and not virtually cyclic
(i.e. of Hirsch length at least $2$).
%%%
%%%
%%%
This class is easily shown to be recursively enumerable. Indeed, 
%%%
%%%
given two finitely presented groups $A=\grp{a_i|r_k(a_i)}$, $B=\grp{b_j|r'_l(b_j)}$
one can enumerate all their extensions by enumerating all the automorphisms of $A$,
and enumerating all product of automorphisms of $A$ whose composition is inner,
and then all presentations of the form $\grp{a_i,b_j |b_ja_ib_j\m=\alpha_j(a_i), r'_l(b_j)=w_l(a_i)}$
where $\alpha_j$ are automorphisms of $A$ such that  $r'_l(\alpha_j)$ coincides with the inner automorphism $\ad_{w_l(a_i)}$.
%%%
%%%
%%%

The following result is a consequence of \cite[Theorem 3]{DG_presenting}.
Since virtually polycyclic groups are residually finite \cite{Hir_soluble3},
it applies in particular to the class $VPC_{\geq2}$.

\begin{thm}\label{thm_find}
Let $\calc$ be a recursively enumerable class of residually finite, finitely presented groups.

Then there exists an algorithm that takes as input a 
 finite presentation  $\grp{S|R}$ of a group $G$ that is hyperbolic relative to some groups in $\calc$,
and which  outputs %%%
a generating set  (as words on $S^{\pm 1}$) and a finite presentation 
of some subgroups $P_1,\dots,P_k<G$,
such that each $P_i$ lies in $\calc$, and $G$ is hyperbolic relative to $P_1,\dots,P_k$.
\end{thm}

\begin{proof}
  Since $\calc$ consists of residually finite groups, $G$ is residually hyperbolic so we have a solution to the word problem
by Corollary \ref{cor_wordpb}. Theorem 3 of \cite{DG_presenting} then directly applies.
\end{proof}

To solve the isomorphism problem of groups that are hyperbolic relative to some subgroups that are not given,
we need these peripheral subgroups to be canonical.

\begin{dfn}
  A group $H$ is \emph{universally parabolic} if for all relatively hyperbolic group $(G,\calp)$
containing a subgroup $H'$ isomorphic to $H$,
$H'$ is parabolic.
\end{dfn}

For example, a finite group, or a virtually cyclic group is not universally parabolic.
By Lemma \ref{lem_tits}, if a subgroup of a relatively hyperbolic group is infinite, not virtually cyclic, and not parabolic, 
then it contains a free subgroup of rank $2$. It follows that groups in $VPC_{\geq 2}$ are universally parabolic.

The point of this definition is that it makes the peripheral subgroups canonical.

\begin{lem}\label{lem_can}
   Consider $(G,\calp),(G,\calp')$ two relatively hyperbolic groups, where 
$\calp,\calp'$ consist of universally parabolic groups.

Then for any isomorphism $\phi:G\ra G'$,
$\phi(\calp)=\calp'$. 
\end{lem}

\begin{proof}
  Indeed, $P\in\calp$ being universally parabolic, there is $P'\in\calp'$ such that
  $\phi(P)\subset P'$. Similarly, there is $P''\in\calp$ such that
  $\phi\m(P')\subset P''$. Since $P,P''$ are in $\calp$ and infinite (because they are universally parabolic),
  this implies $P=P''$. It follows that $\phi(P)=P'$,
  $\phi(\calp)\subset\calp'$ and $\phi(\calp)=\calp'$ by the symmetric
  argument.
\end{proof}

The following theorem therefore applies to the class $\calc=VPC_{\geq 2}$.

\begin{theo}\label{thm_iso_class}
Let $\calc$ be a recursively enumerable class of finitely presented groups that are residually finite and universally parabolic.
Then there is an algorithm that solves the following problem. 

It takes as input a pair of group presentations $G=\grp{S|R}$, $G'=\grp{S'|R'}$, such that $G,G'$ 
are non-elementary hyperbolic relative to some groups in $\calc$  (not given in the input),
and do not split non-trivially over a finite, parabolic or $\Zmax$ subgroup, relative to the parabolic subgroups.

The output is the answer to the question whether they are isomorphic.
\end{theo}

\begin{rem}
  One easily checks that our algorithm is uniform in $\calc$ in the following sense: there exists an algorithm 
that takes as input both a Turing machine enumerating a class $\calc$ of finitely presented groups that are residually finite and universally parabolic, 
and a pair of group presentations, and which solves the isomorphism problem stated in Theorem \ref{thm_iso_class}.
\end{rem}

\begin{proof}
Let $P_1,\dots,P_k<G$ be some groups in $\calc$ such that $G$ is hyperbolic relative to $\{[P_1],\dots,[P_k]\}$.
By Theorem \ref{thm_find}, one can compute a finite generating set and presentation  of some  finitely presented subgroups 
$\Tilde P_1,\dots,\Tilde P_{\Tilde k} <G$
with $\Tilde P_i$ in the class $\calc$, and $G$ is hyperbolic relative to $\{[\Tilde P_1],\dots,[\Tilde P_k]\}$.
Since groups in $\calc$ are universally parabolic, Lemma \ref{lem_can} says that $\{[P_1],\dots,[P_k]\}= \{[\Tilde P_1],\dots,[\Tilde P_{\tilde k}]\}$.
In particular $k=\Tilde k$, and
one has really computed a generating set and presentation of conjugates of $P_1,\dots, P_k$.

Similarly, compute a generating set and a presentation of subgroups $P'_1,\dots P'_{k'}<G'$ in the class $\calc$
such that $G'$ is hyperbolic to $\{[P'_1],\dots,[P'_{k'}]\}$.
Now use the solution to the isomorphism problem in Theorem \ref{theo;isom_algo} to determine whether
$(G,\{[P_1],\dots,[P_k]\})\simeq (G',\{[P'_1],\dots,[P'_{k'}]\})$.
By Lemma \ref{lem_can}, this is the case if and only if $G\simeq G'$.
\end{proof}

\begin{cor}\label{cor_vpc}
 The isomorphism problem is solvable for the class of groups that are non-elementary, relatively hyperbolic with respect 
to virtually polycyclic subgroups, and that do not split non-trivially
over a virtually polycyclic subgroup such that 
 its virtually polycyclic subgroups of Hirsch length $\geq 2$ are conjugate in a factor of the splitting.\qed
\end{cor}

\begin{proof}[Proof of  Corollary \ref{cor_vpc}] %%%
If a group $G$ is relatively hyperbolic with respect to a family $\calp$ of virtually polycyclic groups, then $G$ is still relatively hyperbolic
with respect to the family $\calp_{\geq 2}\subset \calp$ 
 obtained by removing from $\calp$ the finite and virtually cyclic subgroups.
The non-splitting hypothesis asks precisely that $(G,\calp_{\geq 2})$ is rigid. %%%
Thus Theorem \ref{thm_iso_class} applies with $\calc=VPC_{\geq 2}$.
\end{proof}

\subsection{Other applications}
 
  \begin{coro}
    There is an algorithm that takes as input the presentations of two fundamental groups of finite volume manifolds with pinched negative curvature, and  indicates whether they are isomorphic.
  \end{coro}

This is a particular case of Corollary \ref{cor_vpc}
since in this case parabolic subgroups are virtually nilpotent (hence polycyclic-by-finite),
and because, by  \cite[Theorem 1.3(i)]{BeSz_endomorphisms},  there is no  splitting over
elementary subgroups.

If $M,N$  are two complete connected Riemannian
manifolds with pinched negative curvature and finite volume,
then any isomorphism between their fundamental groups is induced by a homotopy equivalence.
In dimension at least 3, any such isomorphism will match the parabolic subgroups of the two manifolds 
(because they are virtually nilpotent, but not virtually cyclic)
so $M$ and $N$ are in fact proper homotopy equivalent.
Now if $M,N$ are of dimension at least $6$, a theorem by Farrell and Jones 
tells us that $M$ and $N$ are then homeomorphic \cite[Corollary 1.1]{FaJo_compact}.
It follows that our algorithm allows to solve the homeomorphism problem for such manifolds (at least from data allowing to compute a presentation of the fundamental group).
\\

We also can apply Theorem \ref{thm_iso_class} for the more exotic class $\calc=\cala\calf$
of (finitely generated free abelian of rank $\geq 2$) by (finitely generated free) groups.
Indeed groups in this class are residually finite (as a  split extension of two finitely generated residually finite groups).
These groups are universally parabolic. Indeed, assume that $H$ is a subgroup of a relatively hyperbolic group, 
and that $H$ contains a normal subgroup $N$ isomorphic to $\bbZ^k$, $k\geq 2$.
Then $N$ has to be parabolic, and by the malnormality %%%
 of peripheral subgroups, this implies that $H$ is parabolic.

  \begin{cor}\label{cor_af} The isomorphism problem is solvable for the class of groups that are non-elementary hyperbolic relative to
    groups in the class $\cala\calf$, and that do not split relative to the parabolic subgroups over a finite, parabolic,
    or $\Zmax$ subgroup.
  \end{cor}

  This corollary might seems surprising given that the isomorphism problem is not solvable in the class $\cala\calf$ \cite{Zimmermann_Klassifikation}.
This shows in particular the necessity of the assumption that the relatively hyperbolic groups considered are non-elementary.

\subsection{A negative result}

We give an example  showing that one cannot enlarge the class of peripheral
groups too much in the previous results.

We base our construction on a wild finitely presented solvable group $P$  with the following properties:
\begin{enumerate}
\item $P=\grp{S|R}$ is finitely presented
\item the center $Z(P)$ of $P$ is free abelian 
\item There exists a recursively enumerable sequence of words $w_i$ on $S^{\pm 1}$, such that $w_i$ represents an elements of $Z(P)$
for all $i$, but such that 
 one cannot decide whether $w_i$ represents the trivial element.
\end{enumerate}

A soluble group with these properties can be constructed using the techniques of \cite{KhSa_survey} (private communication).
See \cite{Kharlampovich_unsolvable81,BGS_algorithmically2} for similar constructions, in which assertion (2) need not hold.
%%%

As usual, we identify the word $w_i$ with the element of $Z(P)$ it represents.

Given the existence of such a group, we prove the unsolvability of the following isomorphism problem.
Recall that if a group has Kazhdan property $(T)$, then it has no non-trivial action on an $\bbR$-tree and no non-trivial splitting at all.

\begin{prop}\label{prop_solvable}
The isomorphism problem between relatively hyperbolic groups
with property $(T)$, and whose peripheral  groups are finitely presented soluble is algorithmically unsolvable,
even if generating sets and presentations of the peripheral subgroups are given in the input.
\end{prop}

First, we construct a group that is hyperbolic relative to $P$.

\begin{lem}\label{lem_exist_RH}
  Given a finitely generated group $P$, there exists a group $G$
  hyperbolic relative to $P$, with Kazhdan property $(T)$. In
  particular $G$ and all its quotients have no non-trivial splitting.

\end{lem}

\begin{proof} 
As this is fairly standard, we only sketch the proof. Write $P=\grp{p_1,\dots,p_k}$, 
and consider $H=\grp{h_1,\dots,h_n}$ a hyperbolic group having property $(T)$, and consider the group $K=H*P$ which is 
hyperbolic relative to $P$.  
 Then for $i=1,\dots,k$ consider an  element $r_i$ of $K$ of the form $p_i w_i(h_1,\dots, h_n)$ so that $r_1,\dots,r_k$ satisfies
a small cancellation property (so that \cite[Thm 2.4]{Osin_SC} can be
applied, ensuring that $P$ embeds in the group $G=K/\ngrp{r_1,\dots r_k}$, and $G$ is hyperbolic relative to $P$). Since $H$ surjects onto $G$,
$H$ inherits property $(T)$.  
\end{proof}

\begin{prop}
  Consider a group $G$ that is relatively hyperbolic relative to a group $P$ as above, with $w_i\in Z(P)$ as above.
Given $k>0$, define $G_i=G/\ngrp{w_i^{ki}}$, and $\bar P_i$ the image of $P$ in $G_i$.

Then  there exists $k>1$ such that all $(G_i,P_i)$ are relatively hyperbolic, 
and one cannot decide whether $G\simeq G_i$, nor whether $(G,P)\simeq (G_i,\bar P_i)$.
\end{prop}

\begin{proof}
  First note that since $w_i$ is central, $\grp{w_i}$ is normal in $P$.
By the Dehn filling theorem \cite{GroMan_dehn,Osin_peripheral}, 
there exists a finite set $S\subset P\setminus\{1\}$ such that for all normal subgroup $N\normal P$
avoiding $S$, $G/\ngrp{N}$ is hyperbolic relative to the image $\bar P$ of $P$, and $\bar P\simeq P/N$.
Since $Z(P)$ is a free abelian group, there exists $k$ such that no element of $S$ is a  $k$-th power in $Z(P)$.
For such a $k$,  $(G_i,\bar P_i)$ is relatively hyperbolic for all $i$.

We claim that $(G,P)\simeq(G_i,\bar P_i)$ if and only if $w_i=1$.
Indeed, if $w_i\neq 1$, the center of $\bar P_i\simeq P_i/\grp{w_i^{ki}}$ contains the torsion element $w_i$.
Since $Z(P)$ is torsion-free, $P\not\simeq \bar P_i$, so $(G,P)\not\simeq(G_i,\bar P_i)$.
Finally we claim that $(G,P)\simeq(G_i,\bar P_i)$ if and only $G\simeq G_i$.
This follows  from Lemma \ref{lem_can} since $\bar P$ and $\bar P_i$ are universally parabolic, because they don't contain a non-abelian free-group,
and are not finite nor virtually cyclic.
\end{proof}

\small
\bibliographystyle{alpha}
%\bibliography{published,unpublished}
\bibliography{Dehn_fill_biblio}

\sc \noindent Fran\c{c}ois Dahmani,  Universit\'e de Grenoble Alpes,  Institut Fourier, F-3800 Grenoble, France.

\tt e-mail:francois.dahmani@univ-grenoble-alpes.fr

% Institut Fourier, Universit\'e de Grenoble,  \\ 100 rue des maths, F-38402 St-Martin d'H\`eres, Cedex France

%\tt e-mail:francois.dahmani@univ-grenoble-alpes.fr

\sc \noindent Vincent Guirardel. 
    Univ Rennes, CNRS, IRMAR - UMR 6625, F-35000 Rennes, France

\tt  e-mail:vincent.guirardel@univ-rennes1.fr

     \end{document}